\documentclass[oneside,english,reqno]{article}
\usepackage[T1]{fontenc}
\usepackage[utf8]{inputenc}
\usepackage{lmodern}
\usepackage{geometry}
\usepackage{hyperref}
\usepackage{color}
\usepackage{amssymb,amsmath,amsthm}
\usepackage{amsthm}
\usepackage{amstext}
\usepackage{amssymb}
\usepackage{graphicx}
\usepackage{esint}

\usepackage{tikz}

\makeatletter \newcommand \listoftodos{\section*{Todo list} \@starttoc{tdo}}
\newcommand\l@todo[2]
{\par\noindent \textit{#2}, \parbox{10cm}{#1}\par} \makeatother

\makeatletter
\numberwithin{equation}{section}
\numberwithin{figure}{section}
\theoremstyle{plain}
\newtheorem{thm}{Theorem}[section]
\newtheorem{fact}[thm]{Fact}
\newtheorem{cor}[thm]{Corollary}
\newtheorem{lem}[thm]{Lemma}
\newtheorem{prop}[thm]{Proposition}

\theoremstyle{remark}
\newtheorem{rem}[thm]{Remark}
\newcommand{\crochet}[1]{{\left\langle #1 \right\rangle}}

\usepackage{color}
  \newcounter{constant} 
  \newcommand{\newconstant}[1]{\refstepcounter{constant}\label{#1}} 
  \newcommand{\useconstant}[1]{C_{\ref{#1}}}

\makeatother

\usepackage{babel}

\begin{document}

\title{Brownian motion and Random Walk above Quenched~Random~Wall}
\author{Bastien Mallein\footnote{LPMA, UPMC and DMA, ENS Paris} \and 
	Piotr Mi\l{}o\'s\footnote{Faculty of Mathematics, Informatics and Mechanics, University of Warsaw, Banacha 2, 02-097 Warszawa, Poland. Email: \texttt{pmilos@mimuw.edu.pl}}}

\maketitle
\textcolor{blue}{}\global\long\def\sbr#1{\left[#1\right] }
\textcolor{blue}{}\global\long\def\cbr#1{\left\{  #1\right\}  }
\textcolor{blue}{}\global\long\def\rbr#1{\left(#1\right)}
\textcolor{blue}{}\global\long\def\ev#1{\mathbb{E}{#1}}
\textcolor{blue}{}\global\long\def\R{\mathbb{R}}
\textcolor{blue}{}\global\long\def\E{\mathbb{E}}
\textcolor{blue}{}\global\long\def\norm#1#2#3{\Vert#1\Vert_{#2}^{#3}}
\textcolor{blue}{}\global\long\def\pr#1{\mathbb{P}\rbr{#1}}
\textcolor{blue}{}\global\long\def\qq{\mathbb{Q}}
\textcolor{blue}{}\global\long\def\aa{\mathbb{A}}
\textcolor{blue}{}\global\long\def\ind#1{1_{#1}}
\textcolor{blue}{}\global\long\def\pp{\mathbb{P}}
\textcolor{blue}{}\global\long\def\cleq{\lesssim}
\textcolor{blue}{}\global\long\def\ceq{\eqsim}
\textcolor{blue}{}\global\long\def\Var#1{\mathrm{Var}(#1)}
\textcolor{blue}{}\global\long\def\dd#1{\textnormal{d}#1}
\textcolor{blue}{}\global\long\def\eqdef{:=}
\textcolor{blue}{}\global\long\def\ddp#1#2{\left\langle #1,#2\right\rangle }
\textcolor{blue}{}\global\long\def\Z{\mathbb{Z}}
\textcolor{blue}{{} }\global\long\def\N{\mathbb{N}}
\textcolor{blue}{{} }
\textcolor{blue}{}\global\long\def\nC#1{\newconstant{#1}}
\textcolor{blue}{}\global\long\def\C#1{\useconstant{#1}}
\textcolor{blue}{}\global\long\def\nC#1{\newconstant{#1}\text{nC}_{#1}}
\textcolor{blue}{}\global\long\def\C#1{C_{#1}}
\textcolor{blue}{}\global\long\def\meas{\mathcal{M}}
\textcolor{blue}{}\global\long\def\cSpace{\mathcal{C}}
\textcolor{blue}{}\global\long\def\pspace{\mathcal{P}}
\begin{abstract}
We study the persistence exponent for the first passage time of a random walk below the trajectory of another random walk. More precisely, let $\{B_n\}$ and $\{W_n\}$ be two centered, weakly dependent random walks. We establish that $\mathbb{P}(\forall_{n\leq N} B_n \geq W_n|W) = N^{-\gamma + o(1)}$ for a non-random $\gamma\geq 1/2$. In the classical setting, $W_n \equiv 0$, it is well-known that $\gamma = 1/2$. We prove that for any non-trivial $W$ one has $\gamma>1/2$ and the exponent $\gamma$ depends only on $\text{Var}(B_1)/\text{Var}(W_1)$.

Our result holds also in the continuous setting, when $B$ and $W$ are independent and possibly perturbed Brownian motions or Ornstein-Uhlenbeck processes. In the latter case the probability decays at exponential rate.
\end{abstract}

\section{Introduction and main results}

Let $\{X_t\}_{t\geq 0}$ be a stochastic process. In many cases of interest there exists $\gamma >0$ such that
\[
  \pp(\forall_{s \leq t} X_s \geq -1) = t^{-\gamma + o(1)}.
\]
The exponent $\gamma$ is often called the persistence exponent associated to the process $X$. The persistence exponent of stochastic processes have received a substantial research attention. We refer to \cite{Aurzada2015} for a review on this subject.

In this paper we are interested in the rate of decay of 
\[
	p(t) := \pp(\forall_{s\leq t} B_s \geq f(t)),
\]
where $B$ is a standard Brownian motion starting from $0$ and $f$ is a continuous function. The case when $f(t) = -1$ is particularly easy following by the reflection principle. More broadly, if $f$ is negative near $0$ and $f(t) = o(t^{1/2-\epsilon})$ for some $\epsilon>0$ then $p(t) = t^{-1/2 + o(1)}$. The situation changes for $f(t) = -1 + \lambda t^{1/2}$ with $\lambda \geq 0$. It turns out that $p(t) = t^{-\theta(\lambda) + o(1)}$ for a function $\theta$, which is strictly increasing in $\lambda$. 

It seems interesting to consider the case when $f$ is a trajectory of a stochastic process. A Brownian motion, considered in this paper, is particularly intriguing as it is of order $t^{1/2}$. Formally, we consider the wall $f$ to be sampled as a Brownian curve $W$ independent of $B$ and kept frozen. We prove there exists a function $\gamma$ such that for any $\beta \in \R$ we have
\[
  \lim_{t \to +\infty} \frac{\log \pr{\forall_{s \leq t} B_s \geq \beta W_s | W}}{\log t} = -\gamma(\beta) \quad \text{ a.s. and in } L^1.
\]
One important result is that $\gamma(\beta)>1/2$ for all $\beta \neq 0$. This means that the conditioning has an strong impact on the asymptotic of this probability, i.e.
\[
  \lim_{t \to +\infty} \frac{ \pr{\forall_{s \leq t} B_s \geq \beta W_s | W}}{\ev{}\rbr{\pr{\forall_{s \leq t} B_s \geq \beta W_s | W}}} = 0 \quad \text{a.s.}
\]
The function $\gamma$ is universal. An analogous result holds for the decay of the probability of a random walk staying above the path of another independent random walk. Further, the assumption of independence can be weakened or, even more generally, we can admit random walks in a time-changing random environment.

Finally, we study an example of strongly ergodic diffusions namely a Ornstein-Uhlenbeck process. We obtain that the probability for an Ornstein-Uhlenbeck process to stay above the path of another Ornstein-Uhlenbeck process decays exponentially fast. The decay is strictly faster than the exponential decay of the expectation of this conditional probability. Our results are presented below in separate subsections further in Section \ref{sec:related_works_motivations} we present related results and motivations.

\subsection{Brownian motion over Brownian motion}\label{sec:BMoverBMresults}

\begin{thm}
\label{thm:BMBasic}
Let $B,W$ be two independent standard Brownian motions. There exists a function $\gamma : \R \to \R$ such that for any $\beta \in \R$, $0\leq a < b \leq +\infty$ and $x > 0$,
\[
  \lim_{t \to +\infty} \frac{\log \pr{\forall_{s\leq t}x+B_{s}\geq\beta W_{s}, B_t - \beta W_t \in (at^{1/2},bt^{1/2})|W}}{\log t} = - \gamma(\beta) \quad \text{a.s. and in } L^p, \ p \geq 1.
\]
Moreover, the function $\gamma$ is symmetric, convex and for any $\beta \neq 0$,
\begin{equation}
  \gamma(\beta)>\gamma(0)=1/2.\label{eq:relevance}
\end{equation}
Consequently, $\gamma$ is strictly increasing and $\lim_{\beta \to +\infty} \gamma(\beta) = +\infty$.
\end{thm}

\begin{rem}
We conjecture that $\beta\mapsto\gamma(\beta)$ is strictly convex and grows at quadratic rate.
\end{rem}

\begin{rem}
	Adhering to the parlance of disordered systems one can interpret (\ref{eq:relevance}) as the relevance of the disorder. Namely, for any $\beta>0$ we have
\[
\lim_{t\to+\infty}\frac{\E\sbr{-\log\pr{\forall_{s\leq t}x+B_{s}\geq\beta W_{s}|W}}}{\log t}>\lim_{t\to+\infty}\frac{-\log\pr{\forall_{s\leq t}x+B_{s}\geq\beta W_{s}}}{\log t} = \frac{1}{2}.
\]
\end{rem}

We contrast this result with the case a random wall with fast decay of correlations. For simplicity we choose an i.i.d. sequence but the result is still valid for other processes such as Ornstein-Uhlenbeck. In this case, the disorder is not relevant, the wall has no impact on the asymptotic behavior of the probability.

\begin{fact}
\label{fact:IIDwall}Let $\cbr{X_{i}}_{i\in\N}$ be an i.i.d. sequence of random variables such that $\E X_{i}=0$ and $\E X_{i}^{2}<+\infty$ and $B$ an independent Brownian motion. We have
\[
  \lim_{N\to+\infty}\frac{\log\pr{\forall_{n\in\cbr{1,\ldots,N}}x+B_{n}\geq X_{n}|\cbr{X_{i}}_{i\in\N}}}{\log N} = - \frac{1}{2}\quad\text{a.s.}
\]
\end{fact}

The result of Theorem \ref{thm:BMBasic} is stable under perturbing the starting condition and the wall.
\begin{thm}
\label{thm:BMExtended}
Let $B,W$ be two independent Brownian motions, $f : \R_+ \to \R$ and $g : \R_+ \to \R_+$ functions such that there exists $\epsilon >0$ verifying
\[
  f(0) = 0, \ \  \lim_{t \to +\infty} \frac{|f(t)|}{t^{1/2-\epsilon}} = 0, \ \ \inf_{t \geq 0} g(t)> 0 \ \text{ and } \ \lim_{t \to +\infty} \frac{\log g(t)}{\log t} = 0. 
\]
For any $\beta \in \R$ and $0 \leq a < b \leq +\infty$ we have 
\begin{multline*}
	\lim_{t\to+\infty}\frac{\log\pr{\forall_{s\leq t}g(t)+B_{s}\geq\beta W_{s}+f(s), B_t - \beta W_t \in (a t^{1/2}, bt^{1/2})|W}}{\log t}\\ =- \gamma(\beta),\quad\text{a.s. and }L^{p},\ p \geq 1.
\end{multline*}
\end{thm}

\subsection{Ornstein-Uhlenbeck process over Ornstein-Uhlenbeck process}

We recall that an Ornstein-Uhlenbeck process $\cbr{X_{t}}_{t\geq0}$ with parameters $(\mu, \sigma)$ is a diffusion fulfilling the stochastic differential equation 
\[
\dd X_{t}=\sigma\dd W_{t}-\mu X_{t}\dd t.
\]

The main result of this section is following. In Remark \ref{rem:BMtoOU} we will explain how this result extends the one of Theorem \ref{thm:BMBasic}.
\begin{thm}
\label{thm:OUresult}Let $X,Y$ be two independent Ornstein-Uhlenbeck processes with the parameters $(\mu_1,1)$ and $(\mu_2,1)$ respectively, where $\mu_{1},\mu_{2}>0$. There exists two functions $\gamma_{\mu_1,\mu_2}$ and $\delta_{\mu_1,\mu_2}$ such that for any $\beta \in \R$, $0 \leq a < b \leq +\infty$, if $X_0>\beta Y_0$ then
\begin{align}
  &\lim_{t\to+\infty}\frac{\log\mathbb{P}\rbr{\forall_{s\leq t}X_{s}\geq\beta Y_{s}, X_t - \beta Y_t \in (a,b)|Y}}{t} = -\gamma_{\mu_{1},\mu_{2}}(\beta) \text{ a.s. and in }L^{p}, \,\, p\geq1,\label{eq:gammaDef}\\
  &\lim_{t\to+\infty}\frac{\log\mathbb{P}\rbr{\forall_{s\leq t}X_{s}\geq\beta Y_{s},X_t - \beta Y_t \in (a,b)}}{t}= -\delta_{\mu_{1},\mu_{2}}(\beta).\label{eq:deltaDef}
\end{align}
The functions $\gamma_{\mu_{1},\mu_{2}}, \delta_{\mu_{1},\mu_{2}}$  are symmetric convex and for any $\beta \neq 0$ we have
\begin{equation}
\gamma_{\mu_{1},\mu_{2}}(\beta)>\delta_{\mu_{1},\mu_{2}}(\beta)>0.\label{eq:mainAim}
\end{equation}
\end{thm}

Observe that the assumption $\sigma_{1}=\sigma_{2}=1$ is non-restrictive. Given an Ornstein-Uhlenbeck process $X$ with parameters $(\mu_1,1)$ the process $\sigma X$ has parameters $(\mu_1,\sigma)$. Setting
\[
  g(\beta;\sigma_{1},\sigma_{2},\mu_{1},\mu_{2}) \eqdef \lim_{t \to +\infty} \frac{-\log\mathbb{P}\rbr{\forall_{s\leq t}X_{s}\geq\beta Y_{s}, X_t - \beta Y_t \in (a,b)|Y}}{t}
\]
where $X,Y$ are two independent Ornstein-Uhlenbeck processes with parameters $(\mu_1,\sigma_1^2)$ and $(\mu_2,\sigma_2^2)$ respectively, we have
\[
  g(\beta;\sigma_{1},\sigma_{2},\mu_{1},\mu_{2})=\gamma_{\mu_{1},\mu_{2}}\rbr{\beta\frac{\sigma_{2}}{\sigma_{1}}}.
\]
Furthermore, using the scaling property of the Brownian motion we obtain
\[
\gamma_{\mu_{1},\mu_{2}}=\mu_{2}\gamma_{\mu_{1}/\mu_{2},1}.
\]
The same relations hold for $\delta_{\mu_{1},\mu_{2}}$.

\begin{rem}\label{rem:BMtoOU}
	 It is well-known that if $W$ is a standard Wiener process then 
	\begin{equation*}
	X_{t}\eqdef xe^{-\mu t}+\frac{\sigma}{\sqrt{2\mu}}e^{-\mu t}W_{e^{2\mu t}-1}\label{eq:OUandBMRelation}
	\end{equation*}
	is an Ornstein-Uhlenbeck process with parameters $(\mu, \sigma)$ starting from $X_{0}=x$. Using this relation one can see that Theorem \ref{thm:OUresult} with $\mu=\mu_{1}=\mu_{2}$ is equivalent to Theorem \ref{thm:BMBasic}. One also checks that
\[
\gamma_{\mu,\mu}(\beta)=2\mu\gamma(\beta),\quad\delta_{\mu,\mu}(\beta)=\mu.
\]
We stress however that the case of different $\mu$'s cannot be expressed in the terms of Theorem \ref{thm:BMBasic}. Moreover, we suspect that $\max(\mu_{1},\mu_{2})>\delta_{\mu_{1},\mu_{2}}(\beta)>\min(\mu_{1},\mu_{2})$.
\end{rem}

\subsection{Random walk in random environment}
\label{sub:Dependent-version-of}

Results analogous to Theorem \ref{thm:BMBasic} hold for random walks. Let $\cbr{B_{n}}_{n\in\N},\cbr{W_{n}}_{n\in\N}$ be two independent random walks. There may exist $n$ such that $\pr{x+B_n \geq W_n}=0$. To resolve this issue we introduce
\begin{equation}
  \label{eq:canStay}
  \mathcal{A}_{x}\eqdef\bigcap_{N\geq0}\cbr{\pr{\forall_{n\leq N}x+B_{n}\geq W_{n}|W}>0}.
\end{equation}
We briefly study these events.
\begin{fact}
\label{fact:nonInfinity}
For any $x\leq x'$ we have $\mathcal{A}_{x}\subset\mathcal{A}_{x'}$ and $\lim_{x\to+\infty}\pr{\mathcal{A}_{x}}=1$. Moreover, the following conditions are equivalent:
\begin{itemize}
  \item For any $x > 0$, $\pr{\mathcal{A}_{x}}=1$.
  \item $\sup S_{B}\geq\sup S_{W}$, where $S_{B},S_{W}$ are respectively the supports of the measures describing $B_{1}$ and $W_{1}$ (we allow both the sides to be infinite).
\end{itemize}
\end{fact}

Now we present an analogue of Theorem \ref{thm:BMBasic} in the random walk settings.
\begin{thm}
\label{thm:BRWBasic}Let $B,W$ be two independent random walks such that $\E B_{1}=\E W_{1}=0$ and suppose that there exists $b>0$ such that $\E e^{b|B_{1}|}<+\infty$ and $\E e^{b|W_{1}|}<+\infty$. Then for any $x>0$ and $0 \leq a < b \leq +\infty$ we have
\begin{multline*}
	\lim_{N\to+\infty}\frac{-\log\pr{\forall_{n\leq N}x+B_{n}\geq W_{n},B_N - W_N \in (a N^{1/2}, b N^{1/2})|W}}{\log N}\\=\begin{cases}
	\gamma\rbr{\sqrt{\frac{\Var{W_{1}}}{\Var{B_{1}}}}} & \text{ on }\mathcal{A}_{x},\\
	+\infty & \text{ on }\mathcal{A}_{x}^{c}.
	\end{cases},\quad\text{a.s.}
\end{multline*}
\end{thm}
We stress that the function $\gamma$ is the same as in theorems in Section \ref{sec:BMoverBMresults}.

Theorem \ref{thm:BRWBasic} can be extended to a more general model of a random walk in random environment that we define now. Let $\mu = \{\mu_n\}_{n \in \N}$ be an i.i.d. sequence with values in the space of probability laws on $\R$. Conditionally on $\mu$ we sample $\{X_n\}_{n\in \N}$ a sequence of independent random variables such that $X_n$ has law~$\mu_n$. Moreover, we set
\[
  S_n \eqdef \sum_{j=1}^n X_j, \ W_n \eqdef -\sum_{j=1}^n \mathbb{E}(X_j|\mu) \text{ and } B_n \eqdef S_n + W_n.
\]
Note that $W$ is a random walk and conditionally on $\mu$ the process $B$ is the sum of independent centred random variables. We make the following standing assumptions:
\begin{description}
  \item[(A1)] We have $\E W_1 =0$, $\Var{W_1}\in[0,+\infty)$ and $\Var{B_{1}}=\E B_{1}^{2}\in (0,+\infty)$.
  \item[(A2)] There exist $C_{1},C_{2}>0$ such that $\E(e^{C_1 |B_1|}|\mu) \leq C_2$ a.s.
  \item[(A3)] There exists $C>0$ such that $\E e^{C |W_1|} < +\infty$.
\end{description}
We introduce a function $f:\N \to \N$ and we extend definition (\ref{eq:canStay}) as follows 
\begin{equation}
	\mathcal{A}_{x}\eqdef\bigcap_{N\geq0}\cbr{\pr{\forall_{n\leq N}x+B_{n}\geq W_{n}+f(n)|W}>0}.\label{eq:canStay2}
\end{equation}

\begin{thm}
\label{thm:BRWBasicExtended}
Let $S$ be a random walk in random environment and $B$, $W$ as described above. Let $f:\N\to \N$ such that $|f(n)| = o(n^{1/2-\epsilon})$ for some $\epsilon>0$. For any $x>0$ and $0 \leq a < b \leq +\infty$ the following limit exists
\begin{multline}
	\lim_{N\to+\infty}\frac{-\log\pr{\forall_{n\leq N}x+S_{n}\geq f(n),S_N \in (aN^{1/2}, bN^{1/2})|\mu}}{\log N}\\=\begin{cases}
	\gamma\rbr{\sqrt{\frac{\Var{W_{1}}}{\Var{B_{1}}}}} & \text{ on }\mathcal{A}_{x},\\
	+\infty & \text{ on }\mathcal{A}_{x}^{c}.
	\end{cases},\quad\text{a.s.}\label{eq:extendedRWoverRWclaim}
\end{multline}
\end{thm}

The previous result holds with some uniformity on the starting position. It is somewhat cumbersome to define an analogue of $\mathcal{A}_x$ in this case. For this reason we state an example with the starting position $x_N\nearrow +\infty$, such that the event becomes trivial by Fact \ref{fact:nonInfinity}.
\begin{thm}
\label{thm:BRWBasicExtended2}Let $S,B$ and $W$ be as above. Let $f:\N\to \N$ such that $|f(n)| = o(n^{1/2-\epsilon})$ for some $\epsilon>0$ and $\cbr{x_n}_n\geq0$ be such that $x_n\nearrow +\infty$ and $x_n = e^{o(\log n)}$. Then for any $0 \leq a < b \leq +\infty$ the following limit exists
\begin{equation*}
	\lim_{N\to+\infty}\frac{\log\pr{\forall_{n\leq N}x_N+S_{n}\geq f(n),S_N \in (aN^{1/2}, bN^{1/2})|W}}{\log N}= -\gamma\rbr{\sqrt{\tfrac{\Var{W_{1}}}{\Var{B_{1}}}}} \text{ a.s.}
\end{equation*}
\end{thm}

\subsection{Related works and motivations}\label{sec:related_works_motivations}

Our result can be understood from various perspectives. One of them is the so-called entropic repulsion. This question was asked in {\cite{Bertacchi:2002aa}} in the context of the Gaussian free field for $d\geq3$. Namely, the authors studied the repulsive effect on the interface of the wall which is a fixed realization of an i.i.d. field $\cbr{\phi_{x}}_{x\in\mathbb{Z}^{d}}$. They observe that the tail of $\phi_{x}$ plays a fundamental role. When it is subgaussian the effect of the wall is essentially equivalent to the wall given by $0$, while when the tail is heavier than Gaussian the interface is pushed much more upwards. It would be interesting to ask an analogous question in our case. By Fact \ref{fact:IIDwall} we know already that the disorder has a negligible effect when $\E X_{i}^{2+\epsilon}<+\infty,$ for $\epsilon>0$. We expect that when $\E X_{i}^{2}=\infty$ the repulsion becomes much stronger.

The paper {\cite{Bertacchi:2002aa}} was followed by \cite{Bertacchi:2003aa} which could be seen as an analogue of our work. Namely, the topic of this paper is a Gaussian free field interface conditioned to be above the fixed realization of another Gaussian free field. The authors obtain the precise estimates for the probability of this event and the entropic repulsion induced by the conditioning.

Theorems \ref{thm:BMBasic} and \ref{thm:OUresult} can be seen as a first step in the study of persistence exponents in random environment. These results give examples of a one-parameter family of persistence exponents.

Similarly, a natural question arising in random walk theory is to study the probability for a random walk to stay non-negative during $n$ units of time. Typically this probability decays as $n^{-1/2}$, which is known as the ballot theorem. Our result stated in Theorem \ref{thm:BRWBasicExtended} provides a version of this result for random walks in random environment. The decay is  $n^{-\gamma}$ for $\gamma\geq 1/2$. Moreover, $\gamma>1/2$ whenever the quenched random walk is not centered.

This perspective was the initial motivation for analyzing the problems in this paper (more precisely the result given in Theorem \ref{thm:BRWBasicExtended}). Such a question arises from studies of extremal particles of a branching random walk in a time-inhomogeneous random environment. In the companion paper~\cite{MM15b} we show that the randomness of the environment has a slowing effect on the position of the maximal particle. Namely, the logarithmic correction to the speed is bigger than in the standard (time-homogenous) case, which is a consequence of (\ref{eq:relevance}).

\subsection{Organization of the paper}

The next section is a collection of preliminary results on the FKG inequality, Ornstein-Uhlenbeck processes and some technical results. In Section \ref{sub:Proof-of-convergence} we use Kingman's theorem to show the convergence \eqref{eq:gammaDef}. We prove \eqref{eq:mainAim} in Section \ref{sec:relevantDisorder} by inferring that the disorder of the wall has an effect on the behaviour of the probability. Section \ref{sec:BMProofs} is devoted to obtaining Theorem~\ref{thm:BMBasic} from Theorem~\ref{thm:OUresult} and generalize it to obtain Theorem \ref{thm:BMExtended}. This last theorem is used in Section \ref{sub:ProofsRW} to study the analogue problem for random walks in random environment. The concluding Section \ref{sec:Disussion-and-Open} contains further discussion and open questions.

\section{Preliminaries and Technical Results}
\label{sub:lemmas}

In this section we list a collection of results that are useful in the rest of the article. We first introduce the so-called FKG inequality for a Brownian motion and an Ornstein-Uhlenbeck process, that states that increasing events are positively correlated.  We then list some integrability facts concerning Ornstein-Uhlenbeck and derive technical consequences.

\subsection{The FKG inequality for Brownian motion and Ornstein-Uhlenbeck processes}

In the proofs we often use the so-called FKG inequality, that we now introduce. For $T \geq 0$, we denote by $\mathcal{C}\eqdef\mathcal{C}([0,T],\R)$, the space of continuous functions, endowed with the uniform norm topology. We introduce a partial ordering $\prec$ on this space: For two $f,g\in\mathcal{C}$, we set
\begin{equation}
  f\prec g\text{ if and only if }\forall_{t\in[0,T]}f(t)\leq g(t).\label{eq:order}
\end{equation}
The FKG inequality is the following estimate, that follows from \cite[Theorem 4 and Remark 2.1]{Barbato:2005aa}.
\begin{fact}
\label{fact:(The-FKG-inequality)}(The FKG inequality) Let $X$ be a Brownian motion or an Ornstein-Uhlenbeck process and $F,G:\mathcal{C}\to\R$ be bounded measurable functions, which are non-decreasing with respect to $\prec$ then 
\begin{equation}
\ev{}\sbr{F(X)G(X)}\geq\ev{}\sbr{F(X)}\ev{}\sbr{G(X)}.\label{eq:fkgFunctions}
\end{equation}
\end{fact}

The result of \cite{Barbato:2005aa} is stated for the Brownian motion only. However, this result is easily transferred to the Ornstein-Uhlenbeck process, as \eqref{eq:OUandBMRelation} preserves the order $\prec$ defined in \eqref{eq:order}. The same reasoning of transferring estimates on the Brownian motion to the Ornstein-Uhlenbeck process holds for the other proofs of the section. Thus to shorten and simplify proofs, we only work with Brownian motion in the rest of the section.

We often use the following corollary of Fact \ref{fact:(The-FKG-inequality)}, stating that increasing events (for the order~$\prec$) are positively correlated.
\begin{cor}
\label{cor:conditionalFKG}
Let $X$ be a Brownian motion or an Ornstein-Uhlenbeck process and $\mbox{\ensuremath{\mathcal{A}},\ensuremath{\mathcal{B}}}$ be increasing events (i.e. such that the functions $1_{\mathcal{A}}$ and $1_{\mathcal{B}}$ are non-decreasing for $\prec$), then
\begin{equation}
  \pr{\mathcal{A}\cap \mathcal{B}}\geq\pr{\mathcal{A}}\pr{\mathcal{B}}.\label{eq:FKG}
\end{equation}
\end{cor}

We also use the following property, sometimes called the strong FKG inequality. 
\begin{lem}
\label{lem:StrongFKG}Let $X$ be a Brownian motion or an Ornstein-Uhlenbeck process, $f,g:\R_{+}\to\R\cup\cbr{-\infty}$ be measurable functions such that $f(t)\geq g(t)$ for all $t\in\R_{+}$. We assume that $\pr{\forall_{t\in[0,T]}X_{t}\geq f(t)}>0$ and $\pr{\forall_{t\in[0,T]}X_{t}\geq g(t)}>0$. The probability distribution $\pr{\cdot|\forall_{t\in[0,T]}X_{t}\geq f(t)}$ stochastically dominates $\pr{\cdot|\forall_{t\in[0,T]}X_{t}\geq g(t)}$ with respect to $\prec$. In other words for any measurable function $h :\R_+ \to \R$, we have
\[
  \pr{\forall_{t \in [0,T]} X_t \geq h(t)|\forall_{t\in[0,T]}X_{t}\geq f(t)} \geq \pr{\forall_{t \in [0,T]} X_t \geq h(t)|\forall_{t\in[0,T]}X_{t}\geq g(t)}.
\]
\end{lem}

\begin{proof}
Let $X$ be a Brownian motion, constructed on the canonical Wiener space $(\mathcal{C},\mathbb{P})$. We assume without loss of generality that $\pr{X_{0}=f(0)}=\pr{X_{0}=g(0)}=0$. Using the Girsanov theorem we observe that $\pr{\forall_{t\in[0,T]}X_{t}\geq f(t)}=\pr{\forall_{t\in[0,T]}X_{t}>f(t)}$. Therefore we can freely exchange the symbols $\geq$ and $>$ whenever convenient.

As $X$ is continuous, note that $\cbr{\forall_{t\in[0,T]}X_{t}\geq f(t)}=\cbr{\forall_{t\in[0,T]}X_{t}\geq\tilde{f}(t)}$,
where $\tilde{f}:[0,T]\mapsto\R$ is given by $\tilde{f}(x)\eqdef\inf_{\omega\in\mathcal{F}}\omega(x)$, where $\mathcal{F}\eqdef\cbr{\omega\in\mathcal{C}:\forall_{t\in[0,T]}\omega(t)>f(t)}$. As the infimum of continuous function, we observe that $\tilde{f}$ is upper semicontinuous. Thus without loss of generality we assume in the rest of the proof that both $f$ and $g$ are upper semicontinuous.

By Baire's theorem there exists a sequence $\cbr{f_{n}}_{n}$ such that $f_{n}\in\mathcal{C}$ and $f_{n}(t)\searrow f(t)$ pointwise. This implies that $\mathcal{A}_{n}\eqdef\cbr{\forall_{t\in[0,T]}X_{t}>f_{n}(t)}$ is an increasing sequence of events and the limiting event is $\bigcup_{n}\mathcal{A}_{n}=\cbr{\forall_{t\in[0,T]}X_{t}>f(t)}$. We have an analogous sequence $\cbr{g_{n}}$ converging pointwise tor $g$. Up to replacing $g_n$ by $\min(f_{n},g_{n})$ we may assume that $g_{n}\prec f_{n}$ for any $n \in \N$.

For any continuous function $f_{n}$ and $\epsilon>0$, there exists a finite set $0\leq t_{1}<t_{2}\ldots<t_{k}\leq T$ such that $\pr{\cbr{\forall_{i\in\cbr{1,\ldots,n}}X_{t_{i}}\geq f(t_{i})}\backslash\cbr{\forall_{t\in[0,T]}X_{t}\geq f(t)}}\leq\epsilon$, by continuity of $X$.

We now assume that the statement of the fact is false. In this case there exists a measurable, bounded non-decreasing function $F : \mathcal{C}\to\R$ and $\epsilon>0$ such that 
\begin{equation}
\E\rbr{F(X)|\forall_{t\in[0,T]}X_{t}\geq f(t)} + \epsilon <\E\rbr{F(X)|\forall_{s\in[0,T]}X_{t}\geq g(t)}.\label{eq:contr0}
\end{equation}
Using the previous arguments, there exists find $n$ and $0\leq t_{1}<\ldots<t_{k}\leq T$ such that 
\begin{equation}
\E\rbr{F|\forall_{i\in\cbr{1,\ldots,k}}X_{t_{i}}\geq f_{n}(t_{i})}<\E\rbr{F|\forall_{i\in\cbr{1,\ldots,k}}X_{t_{i}}\geq g_{n}(t_{i})}.\label{eq:contr}
\end{equation}
Using the same techniques as \cite[B.6]{Giacomin:2001lr} one shows that $\pr{(X_{t_{1}},\ldots,X_{t_{k}})\in\cdot|\forall_{i\in\cbr{1,\ldots,k}}X_{t_{i}}\geq f_{n}(t_{i})}$ stochastically dominates $\pr{(X_{t_{1}},\ldots,X_{t_{k}})\in\cdot|\forall_{i\in\cbr{1,\ldots,k}}X_{t_{i}}\geq g_{n}(t_{i})}$. We notice that conditionally on $X_{t_{i}}=x$ and $X_{t_{i+1}}=y$ the process $\cbr{X_{t}-\sbr{(t_{i+1}-t)x+(t-t)y}/(t_{i+1}-t_{i})}_{t\in[t_{i},t_{i+1}]}$ is a Brownian bridge. Moreover if we condition on the whole vector $(X_{t_{1}},X_{t_{2}},\ldots,X_{t_{k}})$, by the Markov property the brides on the different intervals are independent.

As a consequence, we can construct two processes $X^f$ and $X^g$ such that $X^f$ has the law of $X$ conditionally on $\forall_{i\in\cbr{1,\ldots,k}}X_{t_{i}}\geq f_{n}(t_{i})$ and $X^g$ the law of $X$ conditionally on $\forall_{i\in\cbr{1,\ldots,k}}X_{t_{i}}\geq g_{n}(t_{i})$ such that $X^g \prec X^f$. Indeed, we construct $(X^f_{t_j})$ and $(X^g_{t_j})$ on the same probability space such that $X^f$ dominates $X^g$. We then link $X^f_{t_i}$ with $X^f_{t_{i+1}}$ and $X^g_{t_i}$ with $X^g_{t_{i+1}}$ using the same bridge $\beta^i$ of length $t_{i+1}-t_i$, setting
\[
  \forall s \in [t_i,t_{i+1}], \begin{cases} X^f_s = X^f(t_i) \frac{t_{i+1}-s}{t_{i+1}-t_i} + X^f(t_i) \frac{s-t_{i}}{t_{i+1}-t_i} + \beta^i_{s-t_i} \\  X^g_s = X^g(t_i) \frac{t_{i+1}-s}{t_{i+1}-t_i} + X^g(t_i) \frac{s-t_{i}}{t_{i+1}-t_i} + \beta^i_{s-t_i}.\end{cases}
\]
With this construction, we have $\forall t \leq T, X^g_t \leq X^f_t$, which contradicts \eqref{eq:contr}, thus \eqref{eq:contr0}.
\end{proof}

\subsection{Integrability estimates for Ornstein-Uhlenbeck processes}

We first list a collection of classical estimates for an Ornstein-Uhlenbeck process,
\begin{fact}
\label{fact:OUfact}Let $X$ be an Ornstein-Uhlenbeck process with parameters $\sigma,\mu>0$ starting from $X_{0}=x$.
\begin{enumerate}
  \item The process $X$ is a strong Markov process with an invariant measure $\mathcal{N}\rbr{0,\sigma^{2}/(2\mu)}$. For any $t>0$ the random variable $X_{t}$ is distributed as $\mathcal{N}\rbr{xe^{-\mu t},\frac{\sigma^{2}}{2\mu}(1-e^{-2\mu t})}.$ 
  \item For any $y\in\R$ the right tail of $T_y \eqdef \inf\cbr{t\geq0:X_{t}=y}$ decays exponentially fast.
  \item The process $\cbr{\tilde{X}_{t}}_{t\geq0}$ given by $\tilde{X}_{t}\eqdef X_{t}-e^{-\mu t}X_{0}$ is an Ornstein-Uhlenbeck process with parameters $\sigma,\mu>0$ starting from $\tilde{X}_{0}=0$.
  \item The random variable $\displaystyle M\eqdef\sup_{t\leq1}X_{t}$ has Gaussian concentration i.e. there exist $C,c>0$ such that
\[
  \pr{M>y}\leq C\exp(-cy^{2}),\quad \forall y\geq0.
\]
\end{enumerate}
\end{fact}

The first and third claims follow directly from \eqref{eq:OUandBMRelation}. The fourth claim is a consequence of this fact as well, as $\max_{t \leq 1} X_t \leq x + \frac{\sigma}{\sqrt{2\mu}} \max_{t \leq e^{2\mu}-1} W_s$, where $W$ a Brownian motion. The random variable $\max_{t \leq e^{2\mu}-1} W_s$ has Gaussian concentration by the reflection principle, proving this claim. Finally, the second claim was proved in \cite{NRS}.

We recall a classical bound on the tail estimate of Gaussian random variables.
\begin{fact}
\label{fact:Gaussian}
Let $Z$ be a standard Gaussian random variable and $x>0$. We have
\begin{equation}
\frac{1}{\sqrt{2\pi}}\frac{x}{1+x^{2}}e^{-x^{2}/2}\leq\pr{Z\geq x}\leq\frac{1}{\sqrt{2\pi}}\frac{1}{x}e^{-x^{2}/2}.\label{eq:gaussianTail}
\end{equation}
\end{fact}

We now present a convex analysis result, stating that the probability for a Brownian motion to stay above a curve $f$ is log-convex as a function of $f$.
\begin{lem}
\label{lem:covexTail}Let $X$ be a Brownian motion or Ornstein-Uhlenbeck process and $h_{1},h_{2} : \R\to\R\cup\cbr{-\infty}$ be c\`{a}dl\`{a}g functions such that 
\[
  \pr{\forall_{s\geq0}X_{s}\geq h_{1}(s)}>0,\quad\pr{\forall_{s\geq0}X_{s}\geq h_{2}(s)}>0.
\]
Then the function 
\[
  \begin{array}{rcl}
  [0,1] & \longrightarrow & \R_+\\
  \lambda & \longmapsto & -\log\pr{\forall_{s\geq0}X_{s}\geq\lambda h_{1}(s)+(1-\lambda)h_{2}(s)}
  \end{array}
\]
is convex.
\end{lem}

\begin{proof}
By standard limit arguments it is enough to show that for any $n,N\in\mathbb{N}$ the function 
\begin{equation}
  \begin{array}{rcl}
  [0,1] & \longrightarrow & \R_+\\
  \lambda & \longmapsto &-\log\pr{\forall_{k\leq N}X_{k/n}\geq\lambda h_{1}(k/n)+(1-\lambda)h_{2}(k/n)}
  \end{array}
  \label{eq:beingConvex}
\end{equation}
is convex. We use the Prekopa-Leindler inequality along the lines of the proof below \cite[Theorem 7.1]{Gardner:2002aa}. For $x \in \R^N$, set $H^{\lambda}(x)=d(x)\mathbf{1}_{\cbr{x_{k}\geq\lambda h_{1}(k/n)+(1-\lambda)h_{2}(k/n),\forall {k\leq N}}}(x)$, where we denote by $d$ the joint density of $(X_{1/n},X_{2/n},\ldots,X_{N/n})$. The density $d$ is log-concave i.e for any $\lambda\in(0,1)$ and $x,y\in\R^{N}$ we have $d(\lambda x+(1-\lambda)y)\geq d(x)^{\lambda}d(y)^{(1-\lambda)}$. Similarly 
\[
1_{\forall_{k}\lambda x_{k}+(1-\lambda)y_{k}\geq\lambda h_{1}(k/n)+(1-\lambda)h_{2}(k/n)}\geq\rbr{1_{\forall_{k}x_{k}\geq h_{1}(k/n)}}^{\lambda}\rbr{1_{\forall_{k}y_{k}\geq h_{2}(k/n)}}^{1-\lambda}.
\]
Thus the assumption of the Prekopa-Leindler inequality is fulfilled i.e. 
\[
H^{\lambda}(\lambda x+(1-\lambda)y)\geq\rbr{H^{1}(x)}^{\lambda}\rbr{H^{0}(y)}^{1-\lambda}.
\]
Now \cite[Theorem 7.1]{Gardner:2002aa} implies (\ref{eq:beingConvex}).
\end{proof}

We now give some estimates on the random variable $-\log\pr{\forall_{s\leq1}X_{s}\geq Y_{s}|Y}$, where $X$ and $Y$ are two independent Ornstein-Uhlenbeck processes.
\begin{lem}
\label{lem:boundedMoments}Let $X,Y$ be Ornstein-Uhlenbeck processes.
\begin{enumerate}
  \item Let $C,c>0$ and $x\geq0$. Then there exists $\tilde{C}>0$ such that for any $X_{0}\sim\mathcal{N}(x,c_{1}^{2})$ with $c_{1}\geq c$ and $Y_{0}\sim\mathcal{N}(0,C_{1}^2)$ with $C_{1}\in [0,C]$ we have that
\begin{equation}
  \ev{}\sbr{-\log\pr{\forall_{s\leq1}X_{s}\geq Y_{s}|Y}}\leq\tilde{C}.
  \label{eq:existenceOfTheFirstMoment1}
\end{equation}
  \item Let $X_{0}=x>0$ and $Y_{0}=0$, we set $\rho = \inf\left\{ t \geq 0 : Y_t = 0, \exists s < t : |Y_s| = 1 \right\}$. Then
\begin{equation}
  \ev{}\sbr{-\log\pr{\forall_{s\leq\rho}X_{s}\geq Y_{s}|Y}}<+\infty.
  \label{eq:existenceOfTheFirstMoment2}
\end{equation}
  \item Let $X_{0}=0$, $Y_{0}=0$ and $a,b>0$ then the random variable 
\[
  -\log\pr{\forall_{s\leq1}X_{s}\geq Y_{s}-a,X_{1}\geq Y_{1}+b|Y}
\]
has exponential moments. 
\end{enumerate}
\end{lem}

\begin{proof}
Let $X,Y$ be two independent Ornstein-Uhlenbeck processes of parameters $(\mu_1,\sigma_1), (\mu_2,\sigma_2)$. We first prove the third point. By the FKG property (\ref{eq:FKG}) we have
\[
  -\log\pr{\forall_{s\leq1}X_{s}\geq Y_{s}-a,X_{1}\geq Y_{1}+b|Y}\leq-\log\pr{\forall_{s\leq1}X_{s}\geq Y_{s}-a|Y}-\log\pr{X_{1}\geq Y_{1}+b|Y}.
\]
As $X_1,Y_1$ are two independent Gaussian random variables, the second term in the upper bound has exponential moments by Fact~\ref{fact:Gaussian}. We set $H=-\log\pr{\forall_{s\leq1}X_{s}\geq Y_{s}-a|Y}$. Applying \eqref{eq:OUandBMRelation}, we have
\[
  H\leq-\log\pr{\forall_{s\leq1}B_{t_{1}(s)}\geq\beta|W_{t_{2}(s)}|-a'|W}.
\]
where $B,W$ are Brownian motions, $\beta = \frac{\sigma_2}{\sigma_1}$, $t_{1}(s)=e^{\mu_{1}s}-1$ and $t_{2}(s)=e^{\mu_{2}s}-1$ ($(\mu_{1},\mu_{2}$ are parameters of the Ornstein-Uhlenbeck processes $X$ and $Y$). The constants $a',\beta>0$ can be calculated explicitly but do not matter for the calculations. We denote by
\begin{align*}
  \mathcal{A}_{1} & \eqdef\cbr{\forall_{i\in\mathbb{N}}\forall_{s\in[\frac{3}{4}2^{-i},2^{-i}]}B_{t_{1}(s)}\geq\beta|W_{t_{2}(s)}|-a'},\\
  \mathcal{A}_{2} & \eqdef\cbr{\forall_{i\in\mathbb{N\backslash}\cbr 0}\forall_{s\in[2^{-i},\frac{3}{2}2^{-i}]} B_{t_{1}(s)}\geq\beta|W_{t_{2}(s)}|-a'}.
\end{align*}
Using the FKG property (\ref{eq:FKG}) again we have 
\[
  H\leq-\log\pr{\mathcal{A}_{1}\cap\mathcal{A}_{2}|W}\leq-\log\pr{\mathcal{A}_{1}|W} - \log\pr{\mathcal{A}_{2}|W}.
\]
For $i \in \{1,2\}$ we write $H_{i}=-\log\pr{\mathcal{A}_{i}|W}$. We prove that $H_1$ is exponentially integrable. The same arguments will directly apply to prove that $H_2$ is exponentially integrable. We observe that 
\[
  \mathcal{B}\eqdef\cbr{\forall_{i\in\mathbb{N}}\forall_{s\in[\frac{3}{4}2^{-i},2^{-i}]}B_{t_{1}(s)}-B_{t_{1}(2^{-i}/2)}\geq\beta|W_{t_{2}(s)}-W_{t_{2}(2^{-i}/2)}|-2^{-i/4}a'/8}\subset\mathcal{A}_{1}.
\]
Let $\theta>0$, using the fact that the increments of a Brownian motion are independent we obtain
\[
L(\theta)\eqdef\E\rbr{e^{\theta H_{1}}}\leq\E\rbr{\pr{\mathcal{B}|W}^{-\theta}}= \prod_{i\in\mathbb{N}}L_{i}(\theta),
\]
where 
\[
L_{i}(\theta):=\E\rbr{\pr{\forall_{s\in[\frac{3}{2}2^{-i},2^{-i}]}B_{t_{1}(s)-t_{1}(2^{-i}/2)}\geq\beta|W_{t_{2}(s)-t_{2}(2^{-i}/2)}|-2^{-i/4}a''|W}^{-\theta}},
\]
and $a''\eqdef a'/8$. By the Brownian scaling we get
\[
L_{i}(\theta)=\E\rbr{\pr{\forall_{s\in[\frac{3}{2}2^{-i},2^{-i}]}B_{2^{i}[t_{1}(s)-t_{1}(2^{-i}/2)]}\geq\beta|W_{2^{i}[t_{2}(s)-t_{2}(2^{-i}/2)]}|-2^{i/4}a''|W}^{-\theta}}.
\]
There exist $0<c_{1}<C_{1}$ such that for any $i\in\mathbb{N}$ we have
\[c_{1}\geq2^{i}[t_{1}(\frac{3}{4}2^{-i})-t_{1}(2^{-i}/2)] \quad \text{and} \quad C_{1}\leq2^{i}[t_{1}(2^{-i})-t_{1}(2^{-i}/2)].
\]
We define analogously $0<c_{2}<C_{2}$ associated to $t_{2}$. We set
\[
  M\eqdef\sup_{s\in[c_{2},C_{2}]}\beta W_{s},\quad m\eqdef\inf_{s\in[c_{1},C_{1}]}B_{s}.
\]
With this notation, the estimate becomes 
\[
L_{i}(\theta)\leq\E\rbr{\pr{m\geq M-2^{i/4}a''|M}^{-\theta}}=:\tilde{L}_{i}(\theta).
\]
We have
\[
0\leq\tilde{L}_{i}(\theta)-1=\E\rbr{\frac{1-\pr{m\geq M-2^{i/4}a''|M}^{\theta}}{\pr{m\geq M-2^{i/4}a''|M}^{\theta}}}.
\]
It is well-known that for $x\geq0$ we have $q(x)\eqdef\pr{M>x}\leq C_{3}e^{-c_{3}x^{2}}$ for some $c_{3},C_{3}>0$. We also prove a bound from below for the tail of $m$. Namely, we have
\begin{align*}
  \pr{m>x} & \geq\pr{\cbr{B_{c_{1}}\geq x+5}\cap\cbr{\sup_{s\in[c_{1},C_{1}]}|B_{s}-B_{c_{1}}|\leq5}}\\
  & \geq\pr{B_{c_{1}}\geq x+5}\pr{\sup_{s\in[c_{1},C_{1}]}|B_{s}-B_{c_{1}}|\leq5}\geq C_{4}e^{-c_{4}x^{2}},
\end{align*}
for some $c_{4},C_{4}\geq0$. We combine the estimates to get
\begin{align*}
\tilde{L}_{i}(\theta)-1 & \leq2\theta\pr{M\leq2^{i/4}a''/2}\pr{m\leq-2^{i/4}a''/2}+\pr{M\geq2^{i/4}a''/2}\pr{m\geq0}^{-\theta}\\
 & \quad+\int_{0}^{+\infty}\pr{m\geq y}^{-\theta}\pr{M\geq y+2^{i/4}a''}\dd y.
\end{align*}
There exists $c_5,C_5>0$ such that the first two terms can be bounded from above by $C_{5}e^{-c_{5}i}$ for any $i \in \N$. The last term is bounded in the following way:
\[
  \int_{0}^{+\infty}\pr{m\geq y}^{-\theta}\pr{M\geq y+2^{i/4}a''}\dd y\leq C_{4}^{-\theta}C_{3}\int_{0}^{+\infty}e^{\theta c_{4}y^{2}}e^{-c_{3}(y+2^{i/4})^{2}}\dd y\leq C_{6}e^{-c_{6}i},
\]
for $c_{6},C_{6}>0$, with the last estimate holding under assumption that $\theta$ is small enough (i.e. $\theta <\frac{c_{3}}{c_4}$). We conclude that 
\[
0\leq\tilde{L}_{i}(\theta)-1\leq C_{5}e^{-c_{5}i}+C_{6}e^{-c_{6}i}.
\]
This is enough to claim that $\prod_{i\in\mathbb{N}}\tilde{L}_{i}(\theta)<+\infty$, thus $\prod_{i\in\mathbb{N}}L_{i}(\theta)<+\infty$. We conclude that $H_{1}$ admits exponential moments.

Similar, but simpler, calculations prove that
\begin{equation}
\ev{}\sbr{-\log\pr{\forall_{s\leq1}X_{s}\geq Y_{s}|Y,X_{0}=x,Y_{0}=0}}<+\infty.\label{eq:theFirstAim}
\end{equation}

We now prove (\ref{eq:existenceOfTheFirstMoment2}). By the FKG inequality (\ref{eq:FKG}) we have
\begin{align*}
\ev{}\sbr{-\log\pr{\forall_{s\leq\rho}X_{s}\geq Y_{s}|Y}}\leq & \ev{}\sbr{-\log\pr{\forall_{s\leq1}X_{s}\geq Y_{s}|Y}}\\
 & +\ev{}\sbr{\sum_{i=1}^{\lceil\rho\rceil}-\log\pr{\forall_{s\in[i,i+1]}X_{s}\geq\sup_{s\in[0,\rho]}Y_{s}|Y}}.
\end{align*}
The first term is finite by (\ref{eq:theFirstAim}). To treat the second one we study 
\[
p_{i}(m)\eqdef-\log\pr{\forall_{s\in[i,i+1]}X_{s}\geq m}.
\]
By point 3 of Fact \ref{fact:OUfact} the process $\tilde{X}_{t}\eqdef X_{i+t}-X_{i}e^{-\mu t}$ is an Ornstein-Uhlenbeck process starting from $\tilde{X}_{0}=0$. Therefore we have
\[
p_{i}(m)\leq-\log\pr{X_{i}\geq me^{\mu}+1}-\log\pr{\forall_{s\in[0,1]}\tilde{X}_{s}\geq-1}.
\]
Clearly $-\log\pr{\forall_{s\in[0,1]}\tilde{X}_{s}\geq-1}>-\infty$, using point 1 of Fact \ref{fact:OUfact} and (\ref{eq:gaussianTail}) one easily checks that 
\[
p_{i}(m)\leq C_{9}(m^{2}+1),
\]
for $C_{9}>0$. Recalling that $M=\sup_{s\in[0,\rho]}Y_{s}$ and Fact \ref{fact:decompositon} we conclude 
\[
\ev{}\rbr{\sum_{i=1}^{\lfloor\rho\rfloor}p_{i}\rbr M}\leq\ev{}\rbr{\rho\sbr{C_{9}(M^{2}+1)}}\leq\rbr{\ev{}\rbr{\rho{}^{2}}\ev{}\sbr{C_{9}(M^{2}+1)}^{2}}^{1/2}<+\infty.
\]
The estimate (\ref{eq:existenceOfTheFirstMoment1}) follows by similar calculations and Fact \ref{fact:OUfact}.
\end{proof}

\section{Existence and properties of the function \texorpdfstring{$\gamma$}{g}}
\label{sub:Proof-of-convergence}

In this section, we denote by $X,Y$ two independent Ornstein-Uhlenbeck processes with parameters $(\mu_1,\sigma_1)$ and $(\mu_2,\sigma_2)$ respectively. The main result of the section is the existence of $\gamma_{\mu_1,\mu_2}(\sigma_2/\sigma_1)>0$ such that
\[
  \lim_{t \to +\infty} \frac{-\log\mathbb{P}\rbr{\forall_{s\leq t}X_{s}\geq Y_{s}, X_t - Y_t \in (a,b)|Y}}{t} = \gamma_{\mu_1,\mu_2}(\sigma_2/\sigma_1) \text{ a.s.}
\]
To make notation lighter, we write $\gamma$ instead of $\gamma_{\mu_{1},\mu_{2}}(\sigma_2/\sigma_1)$, as well as $\delta$ instead of $\delta_{\mu_{1},\mu_{2}}(\sigma_2/\sigma_1)$ in the rest of the section. We start proving the annealed part of Theorem \ref{thm:OUresult}.
\begin{lem}
\label{lem:proofDeltaDef}
There exists $\delta > 0$ such that for any $0 \leq a < b \leq +\infty$,
\[
  \lim_{t \to +\infty} \frac{- \log \mathbb{P}\rbr{\forall_{s\leq t}X_{s}\geq Y_{s}, X_t - Y_t \in (a,b)}}{t} = \delta.
\]
\end{lem}

\begin{proof}
This result is a standard application from spectral theory, thus we only present a sketch of the proof. For any $t \geq 0$ and $x,y \in \R$, we set
\[
  u_t(x,y) = \mathbb{P}\rbr{\forall_{s\leq t}X_{s}\geq Y_{s}, X_t - Y_t \in (a,b)}.
\]
We introduce the space $D = \{(x,y) \in \R^2 : x > y\}$ and the operator
\[
  L = \frac{\sigma^2_1}{2} \partial_{x,x} + \frac{\sigma^2_2}{2} \partial_{y,y} - \mu_1 x \partial_x - \mu_2 y \partial_y.
\]
Note that the operator $L$ is an Ornstein-Uhlenbeck operator, which has been the subject of studies in the recent years, we refer to \cite{MOU} and the references therein. By the Feynman-Kac formula we have
\[
  \begin{cases}
    \partial_t u_{t}(x,y) = Lu_t(x,y) & \text{if } (x, y)\in D\\
    u_t(x,y) = 0 &\text{if } (x, y) \not \in D\\
    u_0(x,y) = \ind{x - y \in [a,b]}
  \end{cases}.
\]
Further, let $\nu$ be the measure with density $\exp\rbr{- \frac{\mu_1 x^2}{\sigma_1^2} - \frac{\mu_2 y^2}{\sigma^2_2}} $. We define the scalar product on ${L}^2(D,\nu)$ by
\[
  \crochet{f,g}_\nu = \int_{D} f(x,y) g(x,y) \nu(dxdy).
\]
By integration by parts, for any  $u,v \in \mathcal{C}^\infty(D)$ with compact support, we have
\[
  \crochet{u,L v}_{\nu} = - \frac{\sigma_1^2}{2} \crochet{\partial_x u, \partial_x v}_\nu - \frac{\sigma^2_2}{2} \crochet{\partial_y u, \partial_y v}_\nu = \crochet{v,L u}_\nu.
\]
As a consequence, $L$ can be extended into a negative self-adjoint operator on $\mathcal{H}^1_0(\nu)$. Moreover, this operator is compact, as a consequence, there exists an orthonormal basis $\{h_n\}_{n \in \mathbb{N}}$ of $L^2(D,\nu)$ and a decreasing negative sequence $\{\lambda_n\}_{n \in \mathbb{N}}$ such that
\[
  L h_n = \lambda_n h_n \text{ for any } n \in \N,
\]
and $h_1(x,y)>0$ for any $(x,y) \in D$.

Decomposing $f(x,y) = \ind{x-y \in [a,b]}$ on the basis $h_n$, we obtain
\[
   u_t(x,y) = \sum_{n \in \N}e^{t \lambda_n} h_n(x,y) \crochet{h_n,f}_\nu, \text{ for all } t \geq 0.
\]
In particular, this yields $\lim_{t \to +\infty} \frac{1}{t} \log u_t(x,y) = \lambda_1 <0$ uniformly on compact sets, concluding the proof.
\end{proof}

\subsection{Path decomposition}
\label{sub:Path-decomposition}

In this section, we present a decomposition of the path $Y$ into large excursions. This decomposition is used both in proofs of \eqref{eq:gammaDef} and \eqref{eq:mainAim}. We define the random variables $\cbr{\tau_{i}}_{i\geq0}, \cbr{\rho_{i}}_{i\geq0}$ such that $\rho_0=0$ and
\[
  \rho_{i+1}\eqdef
  \inf\cbr{t\geq\rho_{i}:Y_{t}=0\text{ and }\exists_{s\in(\rho_{i},t)}|Y_{s}|=1}
  \quad \text{and} \quad \tau_{i}\eqdef\sup\cbr{t<\rho_{i+1}:Y_{t}=0}.
\]
We also define $r_{i}\eqdef\rho_{i+1}-\rho_{i}$ and denote 
\begin{equation}
  Y^{i}(t)\eqdef Y_{t+\rho_{i}},\quad t\in[0,r_{i}].\label{eq:Yi}
\end{equation}

\begin{figure}[!ht]\label{fig:notation}
\caption{Notation used in the paper.}
\includegraphics[scale=0.42]{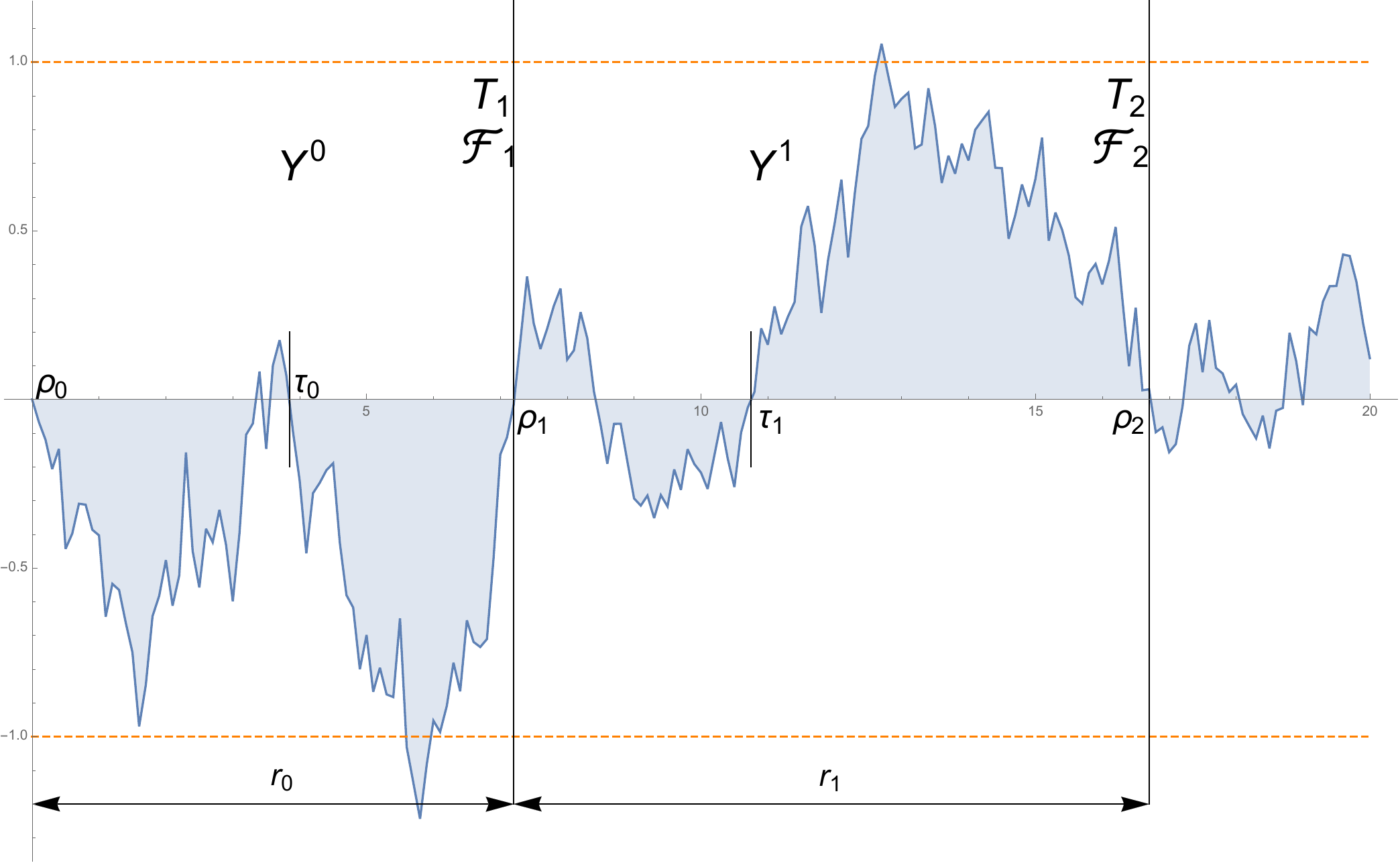}
\end{figure}

\begin{rem}
Note that $\rho_{i}$ is a stopping time (contrary to $\tau_{i}$). The precise definition of $\rho_{i}$ and $\tau_{i}$ are not important. What matters for our proofs is that on the interval $[\tau_{i},\rho_{i+1}]$ the process performs a ``macroscopic'' excursion which is symmetric around $0$, and that $\rho_i$ has a finite mean.
\end{rem}

\begin{fact}
\label{fact:decompositon}
The sequence $\cbr{(Y^{i},r_{i})}_{i\geq0}$ is i.i.d. and the random variables $r_{i}$ and $M^{i}\eqdef\sup_{s\leq r_{i}}Y^{i}(s)$ have tails which decay exponentially fast.
\end{fact}

\begin{proof}
The first statement follows by the fact that $Y_{\rho_{i}}=0$ and the strong Markov property applied to~$Y$. We define $\tilde{\rho}\eqdef\inf\cbr{t\geq0:|Y_{t}|=1}$ then $\rho_{1}=\inf\cbr{t\geq\tilde{\rho}:Y_{t}=0}$. 

By point 2 of Fact \ref{fact:OUfact} both $\tilde{\rho}$ and $\rho_{1}-\tilde{\rho}$ have exponential tails, thus $r_{1}=\rho_{1}$ has an exponential tail. For $x\geq0$ we have 
\[
  \pr{M^{i}\geq x}\leq\pr{\sup_{s\leq x}Y^{i}(s)\geq x}+\pr{r_{i}\geq x}\leq\sum_{k=1}^{\lceil x\rceil}\pr{\sup_{s\in[k-1,k]}Y^{i}(s)\geq x}+\pr{r_{i}\geq x}.
\]
We deduce that $M^{i}$ has an exponential tail using point 4 of Fact \ref{fact:OUfact}. 
\end{proof}

\subsection{A modified version of Theorem \ref{thm:OUresult}}

In this section, we denote by $\pp_x$ the law of $(X,Y)$ such that $X_0=x$ and $Y_0=0$. In a first time, we study the asymptotic behaviour of $\log  \pp(\forall_{u\in[0,\rho_{n}]}X_{u}\geq Y_{u}|Y)$ as $n \to +\infty$, using Kingman's subadditive ergodic theorem.
\begin{lem}
\label{lem:existsLimit}
We assume that $Y_0=0$. For any $0 < a < b \leq +\infty$, there exists $\tilde{\gamma}_{a,b}$ such that
\[
  \lim_{n \to +\infty} \frac{-\log \inf_{x\in (a,b)} \pp_x(\forall_{u\in[0,\rho_{n}]}X_{u}\geq Y_{u}, X_{\rho_n}\in (a,b)|Y)}{\log n} = \tilde{\gamma}_{a,b} \quad \text{a.s. and in } L^1.
\]
\end{lem}

\begin{proof}
Let $0<a<b \leq +\infty$, we set $I=(a,b)$. For any $0 \leq m < n$, we set
\begin{equation}
  p_{m,n}\eqdef\inf_{x\in (a,b)} \pp(\forall_{u\in[\rho_{m},\rho_{n}]}X_{u}\geq Y_{u}, X_{\rho_{n}}\in I|Y,X_{\rho_{m}}=x) \label{eq:pnm}
\end{equation}
and $q_{m,n} = -\log p_{m,n}$. By Markov inequality applied to $X$, we have
\[
  p_{m,n} = \inf_{x\in (a,b)} \pp_x(\forall_{u\in[0,\rho_{n}-\rho_m]}X_{u}\geq Y_{u+\rho_m}, X_{\rho_{n}-\rho_m}\in I|Y)
\]
We observe that the process $X$ under $\pp_x$ has the same law as $(X_t + xe^{-\mu_t})$ under $\pp_0$. Therefore, from the FKG inequality (\ref{eq:FKG}), if $b=+\infty$ then the minimal value of $p_{m,n}$ is attained at $x=a$.

We prove that $\cbr{q_{m,n}}_{n>m\geq1}$ fulfils the assumptions of Kingman's subadditive ergodic theorem as stated in \cite[Theorem 9.14]{Kallenberg:1997fk}. By the Markov property, as $Y_{\rho_{n}}=0$ for any $1 \leq m<n$ we have 
\[
  p_{0,n}=p_{0,m}\pp(\forall_{u\in[\rho_{m},\rho_{n}]}X_{u}\geq Y_{u}, X_{\rho_{n}}\in I|Y,\forall_{u\in[0,\rho_{m}]}X_{u}\geq Y_{u},X_{\rho_{m}}\in I)\geq p_{0,m}p_{m,n},
\]
thus $q_{0,n}\leq q_{0,m}+q_{m,n}$, which is the subadditivity condition \cite[(9.9)]{Kallenberg:1997fk}.

We fix $k \geq 1$. We recall that $\cbr{Y^l}_{l \geq 0}$ is i.i.d. Consequently the sequence
\begin{equation}
  \cbr{q_{lk,(l+1)k}}_{l\geq0}\label{eq:qkln}
\end{equation}
is i.i.d. and condition \cite[(9.7)]{Kallenberg:1997fk} is fulfilled. Further, condition \cite[(9.8)]{Kallenberg:1997fk} follows by the fact that the process $\cbr{Y_{t+\rho_{k}}}_{t\geq0}$ is an Ornstein-Uhlenbeck process distributed as $Y$. As $q_{0,n}\geq 0$; Lemma \ref{lem:boundedMoments} implies that $\E q_{0,1}<+\infty$ thus \cite[Theorem 9.14]{Kallenberg:1997fk} applies and 
\begin{equation}
  \lim_{n\to+\infty}\frac{-\log p_{0,n}}{n}=\lim_{n\to+\infty}\frac{q_{0,n}}{n}=:\tilde{\gamma}_{a,b},\quad\text{a.s.}\:\text{and }L^{1}.\label{eq:basicConvergence}
\end{equation}
The constant $\tilde{\gamma}_{a,b}$ is non-random since (\ref{eq:qkln}) is ergodic.
\end{proof}

Now we prove that the constant $\tilde{\gamma}$ does not depend on $(a,b)$.
\begin{lem}
There exists $\tilde{\gamma}>0$ such that for any $0< a < +\infty$ we have $\tilde{\gamma} = \tilde{\gamma}_{a,+\infty}$.
\end{lem}

\begin{proof}
For any $a \geq 0$ and $x > 0$, we write
\[
  p_n(x,a) \eqdef \pp_x(\forall_{u\in[0,\rho_{n}]}X_{u}\geq Y_{u}, X_{\rho_{n}}>a|Y),
\]
and accordingly $q_n(x,a)\eqdef -\log p_n(x,a)$. We prove that 
\begin{equation}
	\lim_{n\to+\infty} \frac{q_n(x,a)}{n} = \tilde{\gamma},\quad\text{a.s.}\:\text{and }L^{1}, \label{eq:limitAll}
\end{equation}
exists and is independent of $x>0,a\geq 0$. Fix $x>0$, by \eqref{eq:basicConvergence}, we know that
\[
  \lim_{n \to +\infty} \frac{q_n(x,x)}{n} = \tilde{\gamma}_{x,+\infty},\quad\text{a.s.}\:\text{and }L^{1},
\]
as the minimum in \eqref{eq:pnm} is attained in $x=a$. We prove that $p_n(x,0,+\infty)$ behaves similarly. As $p_n(x,x)\leq p_n(x,0,+\infty)$, we have
\begin{align*}
  0\leq d_{n} & :=\frac{q_n(x,x,+\infty)}{n}-\frac{q_n(x,0,+\infty)}{n}=-\frac{1}{n}\log\frac{p_n(x,x,+\infty)}{p_{n}(x,0,+\infty)}\\
  & =-n^{-1}\log\pp_x\sbr{X_{\rho_{n}}\geq x|\forall_{u\in[\rho_{0},\rho_{n}]}X_{u}\geq Y_{u},Y}\\
  &\leq-n^{-1}\log\pp_x\sbr{X_{\rho_{n}}\geq x|Y},
\end{align*}
by the FKG inequality. We conclude easily that $d_n \to 0$ a.s. and in $L^1$. By a simple monotonicity argument we conclude that convergence \eqref{eq:limitAll} holds for any pair $(x,a)$, when $x>0$ and $a\in[0,x]$ and the limit depends only on $x$.

We now fix $x_1>x_2>0$, we have $p_n(x_1,0,+\infty) \leq p_n(x_2,0,+\infty)$. On the other hand 
\begin{equation}\label{eq:conditioning}
  \frac{q_n(x_2,0,+\infty)}{n}\leq\frac{-\log\pp_{x_2}\sbr{\forall_{s\leq\rho_{1}}X_{s}\geq Y_{s},X_{\rho_{1}}\geq x_{1}|Y}}{n} + \frac{q_{n-1}(x_1,0,+\infty)}{n}.
\end{equation}
This proves that $\tilde{\gamma} = \tilde{\gamma}_{x,+\infty}$ does not depend of $x$.
\end{proof}

\begin{lem}\label{lem:limitIndependence}
For any $0 < a < b \leq +\infty$, we have $\tilde{\gamma} = \tilde{\gamma}_{a,b}$
\end{lem}

\begin{proof}
Using the previous lemma, we set $\tilde{\gamma} = \gamma_{a,+\infty}$ for any $a>0$. To show the claim it is enough to prove that for any $b<+\infty$ the limit cannot be smaller. We define $n_0 = \lceil n -C_1\log n \rceil$ for $C_1>1$ to be fixed later and $n_1 = n-1$. Using the Markov property we decompose
\[
  \pp_x(\forall_{u\leq\rho_{n}}X_{u}\geq Y_{u}, X_{\rho_{n}} \in (a,b)|Y) \geq p_1(n) p_2(n) p_3(n),
\]
where
\begin{align*}
  p_1(n) &\eqdef \inf_{x\in (a,b)}\pp_x(\forall_{u\leq\rho_{n_0}}X_{u}\geq Y_{u}, X_{\rho_{n_0}} \in (a ,n)|Y),\\
  p_2(n) &\eqdef \inf_{x\in[a,n]}\pp(\forall_{u\in[\rho_{n_0},\rho_{n_1}]}X_{u}\geq Y_{u}, X_{\rho_{n_1}} \in (a,\log n)|Y,X_{\rho_{n_0}}=x),\\
  p_3(n) &\eqdef \inf_{x\in [a,\log n]} \pp(\forall_{u\in[\rho_{n_1},\rho_{n}]}X_{u}\geq Y_{u}, X_{\rho_{n}} \in (a,b)|Y,X_{\rho_{n_1}}=x).
\end{align*}
We prove that 
\begin{equation}
  \lim_{n\to+\infty } \frac{-\log p_1(n)}{n} = \tilde{\gamma}, \quad \lim_{n\to+\infty } \frac{-\log p_2(n)}{n} = 0,  \quad \liminf_{n\to+\infty } \pp \left( -\log p_3 \leq n^{1/2}\right)>0,
  \label{eq:smallAim2}
\end{equation}
where the first two convergences hold in probability. This limit and \eqref{eq:basicConvergence} imply the claim of the lemma. We first treat the second convergence. We have
\begin{align*}
  p_2(n) &\geq \inf_{x\in[a,n]}\pp_x(\forall_{u\in[0,\rho_{n_1}-\rho_{n_0}]}X_{u}\geq Y_{\rho_{n_0}+u}, X_{\rho_{n_1}-\rho_{n_0}}\geq a|Y) - \sup_{x\in [a,n]} \pp_x( X_{\rho_{n_1}-\rho_{n_0}}\geq \log n|Y)\\
  & \geq \pp_a(\forall_{u\in[0,\rho_{n_1}-\rho_{n_0}]}X_{u}\geq Y_{u+\rho_{n_0}}, X_{\rho_{n_1}-\rho_{n_0}}\geq a|Y) - \pp_n( X_{\rho_{n_1}-\rho_{n_0}}\geq \log n|Y),
\end{align*}
using the Markov property and the FKG inequality. Using \eqref{eq:basicConvergence} we get that the logarithm of the first term is $\approx C_1 \tilde{\gamma} \log n$. Now we fix $C_1$ large enough so that $\E{} X_{\rho_{n_1} -\rho_{n_0}} \leq 1$. This can be established using property 3 of Fact \ref{fact:OUfact}. Further property 1 implies that $\log \pp_n( X_{\rho_{n_1}-\rho_{n_0}}\geq \log n|Y) \leq - c_1 \log^2 n$, for some $c_1>0$. We conclude that for large $n$ the second term is negligible and thus for some $C>0$
\[
	\lim_{n\to +\infty}\pp\rbr{\frac{-\log p_2(n)}{ \log n} \geq C} = 0.
\]
%
This yields the second convergence in \eqref{eq:smallAim2}. An analogous proof gives the first one. For the last one we consider an event $\mathcal{A}_n \eqdef \{\rho_{n}-\rho_{n_1} \in [1,2], \sup_{s\in [\rho_{n_1}, \rho_n]} |Y_s| \leq a/2 \}$. Clearly,
\[
  p_3(n) \geq \inf_{x\in [1,\log n]} \pp_x(X_{\rho_{n}-\rho_{n,1}} \in (a,b)|Y) \pp_a(\forall_{u\in[0,\rho_{n}-\rho_{n_1}]}X_{u}\geq Y_{u+\rho_{n_1}}|Y,X_{\rho_{n}-\rho_{n_1}}=a).
\]
Conditionally on $\mathcal{A}_n$ the second term is bounded from below by a constant and the first one by $\exp(-(\log n)^3)$. We conclude that for large $n$ there is $\pp(-\log p_3 \leq n^{1/2}) \geq \pp(\mathcal{A}_n)$. This finishes the proof as the right-hand side is non-zero and does not depend on $n$.
\end{proof}

Finally, we prove the limit in Lemma \ref{lem:existsLimit} holds for any starting position.
\begin{lem}
For any $x > y$  and $0\leq a < b \leq +\infty$, we have
\[
  \lim_{n \to +\infty} \frac{-\log \pp_x(\forall_{u\in[0,\rho_{n}]}X_{u}\geq Y_{u}, X_{\rho_{n}} \in (a,b)|Y, Y_0 = y)}{\log n} = \tilde{\gamma} \quad \text{a.s and in } L^1,
\]
\end{lem}

\begin{proof}
We prove this result assuming $Y_0=0$, the case $Y_0 \neq 0$ being treated in a similar way. We write
\[
  q_n(x,a,b) = -\log \pp_x(\forall_{u\in[0,\rho_{n}]}X_{u}\geq Y_{u}, X_{\rho_{n}} \in (a,b)|Y)
\]
Using the two previous lemmas, we have
\[
  \limsup_{n\to_\infty} \frac{q_n(x,a,b)}{n} \leq \lim_{n\to+\infty} \frac{\sup_{x\in(a,b)} q_n(x,a,b)}{n} = \tilde{\gamma},
\]
as $\sup_{x\in(a,b)} q_n(x,a,b)=q_{0,n}$. Similarly, for any $x \geq a$, we have
\[  \liminf_{n \to +\infty} \frac{q_n(x,a,b)}{n} \geq \liminf_{n \to +\infty} \frac{q_n(x,a,+\infty)}{n}
  \geq \liminf_{n \to +\infty} \frac{q_n(a,a,+\infty)}{n} \geq \tilde{\gamma},
\]
by the FKG inequality. Finally, using a reasoning similar to \eqref{eq:conditioning}, a similar inequality holds for $x \leq a$.
Consequently, the convergence 
\begin{equation}
	\lim_{n\to+\infty} \frac{q_n(x,a,b)}{n} = \tilde{\gamma},\quad\text{a.s.}\:\text{and }L^{1}, \label{eq:limitAll2}
\end{equation}
holds for any $a,b,x$.
\end{proof}

\subsection{Existence and basic properties of \texorpdfstring{$\gamma$}{g}}

We now prove that \eqref{eq:gammaDef} holds. To this end we state an auxiliary fact, whose proof is postponed to the end of the section.
\begin{fact}
\label{fact:uniformIntegrability}
For any $C\geq0$, $a\geq C$ and $b\in(a,+\infty]$ the family of random variables $\cbr{H_{t}}_{t\geq0}$ defined by 
\begin{equation}
  H_{t}\eqdef\frac{-\log\mathbb{P}\rbr{\forall_{s\leq t}X_{s}\geq Y_{s}+C,X_t-Y_t\in (a,b)|Y}}{t}.
  \label{eq:Ht}
\end{equation}
is $L^{p}$-uniformly integrable for any $p\geq1$.
\end{fact}
\begin{proof}
Without loss of generality it is enough to work with integer times and assume that $a>C$. Denoting the probability in \eqref{eq:Ht} by $p_t$ we estimate
\begin{align*}
  -\log p_n & \leq-\log\mathbb{P}\rbr{\forall_{s\leq n}X_{s}\geq Y_{s}+C,\forall_{k\in\cbr{1,\ldots,n}}X_{k}-Y_k\in(a,b)|Y}\\
  & \leq-\log\mathbb{P}\rbr{\forall_{s\in[0,1]}X_{s}\geq Y_{s}+C,X_{1}-Y_1 \in (a,b)|Y}+\sum_{k\in\cbr{1,\ldots,n-1}}q_{k},
\end{align*}
where $q_{k}\eqdef-\log\sbr{\inf_{x\in(a,b)}\mathbb{P}\rbr{\forall_{s\in[k,k+1]}X_{s}\geq Y_{s}+C,X_{k+1}-Y_{k+1}\in (a,b) |Y,X_{k}-Y_k=x}}$.

By the Markov property the random variables $\cbr{q_{k}}_{k\geq1}$ are independent and identically distributed thus, by Lemma \ref{lem:boundedMoments}, the sequence $\cbr{\frac{1}{n}\sum_{k=1}^{n}q_{k}}_{n}$ is $L^{p}$-uniformly integrable. Further, the proof follows by standard arguments. 
\end{proof}

\begin{lem}
\label{lem:proofGammaDef}
For any $0 \leq a < b \leq +\infty$ and $X_0 > Y_0$, we have
\[
  \lim_{t \to +\infty} \frac{-\log\mathbb{P}\rbr{\forall_{s\leq t}X_{s}\geq Y_{s},X_t-Y_t\in (a,b)|Y}}{t} = \frac{\tilde{\gamma}}{\mathbb{E}(r_1)} \quad \text{a.s. and in } L^1.
\]
\end{lem}

Consequence of this lemma, we set $\gamma = \frac{\tilde{\gamma}}{\mathbb{E}(r_1)}$.

\begin{proof}
Let $m(t)\eqdef\lfloor t/\ev{}\rho_{1}-t^{2/3}\rfloor$ and $M(t)\eqdef\lfloor t/\ev{}\rho_{1}+t^{2/3}\rfloor$ and 
\[
  \mathcal{A}_{t}\eqdef\cbr{t\in[\rho_{m(t)},\rho_{M(t)}]}.
\]
Clearly, $\rho_{n}=\sum_{k=0}^{n-1}r_{k}$, using Fact \ref{fact:decompositon} one checks that $1_{\mathcal{A}_{t}}\to1$ a.s. By Fact \ref{fact:uniformIntegrability} it follows that
\begin{equation}
  \lim_{t\to+\infty}H_{t}1_{\mathcal{A}_{t}^{c}}=0,\quad\text{a.s and }L^{p}.\label{eq:convergenceTo0}
\end{equation}
By \eqref{eq:limitAll2} we have
\begin{equation}
  \liminf_{t\to+\infty} H_{t}
  \geq \lim_{t\to+\infty}1_{\mathcal{A}_{t}}\frac{-\log\mathbb{P}\rbr{\forall_{s\leq\rho_{m(t)}}X_{s}\geq Y_{s}|Y}}{m(t)}\frac{m(t)}{t} = \frac{\tilde{\gamma}}{\ev{}r_{1}},\quad\text{a.s and }L^{p}.
  \label{eq:xyz-1}
\end{equation}
The bound from above is slightly more involved 
\[
	1_{\mathcal{A}_{t}} H_{t}\leq 1_{\mathcal{A}_{t}}\frac{-\log\mathbb{P}\rbr{\forall_{s\leq\rho_{m(t)}}X_{s}\geq Y_{s},\forall_{s\in[\rho_{m(t)},\rho_{M(t)}]}X_s-Y_s\in (a,b)|Y}}{M(t)}\frac{M(t)}{t}.
\]
Let us denote the probability in the expression above by $p$. We fix $a',b'$ such that $a<a'<b'<b$ and use the Markov property
\begin{multline*}
  \log p\geq \log\mathbb{P}\rbr{\forall_{s\leq\rho_{m(t)}}X_{s}\geq Y_{s},X_{\rho_{m(t)}}-Y_{\rho_{m(t)}}\in (a',b')|Y} \\
  +\log\sbr{\inf_{x\in [a',b']} \log \mathbb{P}\rbr{\forall_{s\in[\rho_{m(t),\rho_{M(t)}}]}X_s-Y_s\in (a,b)|Y,X_{\rho_{m(t)}}=x}}.
\end{multline*}
It is easy to check that the second term divided by $t$ converges to $0$ (which essentially follows by the fact that $M(t)-m(t) = o(t)$). Thus using \eqref{eq:limitAll2} again we obtain
\[
  \liminf_{t\to+\infty} H_{t} \leq \frac{\tilde{\gamma}}{\ev{}r_{1}},\quad\text{a.s and }L^{p}.
\]
This together with \eqref{eq:xyz-1} concludes the proof of (\ref{eq:gammaDef}). 
\end{proof}

We study the properties of the exponent $\gamma$ as a function of a constant $\beta$ by which the path $Y$ is multiplied. More precisely, with a slight abuse of notation, we set for $\beta \in \R$
\[
  \gamma(\beta) = \lim_{t \to +\infty} \frac{1}{t} \log \pr{X_s \geq \beta Y_s, s \leq t | Y},
\]
again without denoting the dependence in $\mu_1,\mu_2,\sigma_1,\sigma_2$, to avoid cumbersome notation.

\begin{lem}
\label{lem:fungamma}
The function $\gamma$ is symmetric and convex.
\end{lem}

\begin{proof}
As the law of $Y$ is symmetric, $\gamma$ is symmetric. In effect, for any $t \geq 0$ we have
\[
 \pr{X_s \geq \beta Y_s, s \leq t | Y} \overset{(d)}{=}  \pr{X_s \geq -\beta Y_s, s \leq t | Y},
\]
therefore $\lim_{t \to +\infty} \frac{1}{t} \log \pr{X_s \geq \beta Y_s, s \leq t | Y} \overset{(d)}{=} \lim_{t \to +\infty} \frac{1}{t} \log \pr{X_s \geq -\beta Y_s, s \leq t | Y}$. We conclude that $\gamma(\beta) = \gamma(-\beta)$.

To prove convexity we use Lemma \ref{lem:covexTail}. To this end we fix $t>0$ and $\lambda\in(0,1)$. Applied conditionally on $Y$ the lemma implies that for any $t \geq 0$, almost surely
\[
  -\frac{\log\pp(\forall_{s\leq t}X_{s}\geq(\lambda a+(1-\lambda)b)Y_{s}|Y)}{t} \leq -\lambda\frac{\log\pp(\forall_{s\leq t}X_{s}\geq aY_{s}|Y)}{t}-(1-\lambda)\frac{\log\pp(\forall_{s\leq t} X_{s}\geq bY_{s}|Y)}{t}.
\]
Taking $t\to+\infty$ we obtain $\gamma(\lambda a+(1-\lambda)b)\leq\lambda\gamma(a)+(1-\lambda)\gamma(b).$
\end{proof}

Using the same arguments, one obtains easily the following result.
\begin{lem}
\label{lem:fundelta}
For $\beta \in \R$, we set
\[
  \delta(\beta) = \lim_{t \to +\infty} \frac{1}{t} \log \pr{\forall_{s \leq t} X_s \geq \beta Y_s},
\]
the function $\delta$ is symmetric and convex.
\end{lem}

\begin{proof}
We recall that by Lemma \ref{lem:covexTail}, the function
\[
  \beta \mapsto - \log \pp(\forall_{s \leq t}, X_s \geq \beta Y_s|Y)
\]
is a.s. a convex function. Thus, for any $\lambda \in [0,1]$ and $\beta,\beta'$ we have
\[
  \pp(\forall_{s \leq t}, X_s \geq (\lambda\beta + (1-\lambda) \beta') Y_s|Y) \leq  \pp(\forall_{s \leq t}, X_s \geq \beta Y_s|Y)^\lambda  \pp(\forall_{s \leq t}, X_s \geq \beta' Y_s|Y)^{1-\lambda}.
\]
Consequently, using Holder inequality, we obtain
\[
  \ev{}( \pp(\forall_{s \leq t}, X_s \geq (\lambda\beta + (1-\lambda) \beta') Y_s|Y)) \leq \pp(\forall_{s \leq t}, X_s \geq \beta Y_s)^\lambda\pp(\forall_{s \leq t}, X_s \geq \beta' Y_s)^{1-\lambda},
\]
concluding the proof.
\end{proof}

\section{Relevance of the disorder}
\label{sec:relevantDisorder}

By Lemmas \ref{lem:proofDeltaDef}, \ref{lem:proofGammaDef}, \ref{lem:fungamma} and \ref{lem:fundelta}, there exist two convex symmetric functions $\gamma_{\mu_1,\mu_2}$ and $\delta_{\mu_1,\mu_2}$ such that \eqref{eq:gammaDef} and \eqref{eq:deltaDef} both hold. Therefore, the only thing left to prove Theorem \ref{thm:OUresult} is the strict inequality \eqref{eq:mainAim}.

Observe that for any fixed $t>0$, by Jensen's inequality we have
\[
\ev{}\sbr{-\log\mathbb{P}\rbr{\forall_{s\leq t}X_{s}\geq \beta Y_{s}|Y}}>-\log\mathbb{P}\rbr{\forall_{s\leq t}X_{s}\geq \beta Y_{s}},
\]
which implies $\gamma_{\mu_{1},\mu_{2}}(\beta)\geq\delta_{\mu_{1},\mu_{2}}(\beta)$ for any $\beta \in \R$. Obtaining the strict inequality is a much harder result. We recall the path decomposition from Section \ref{sub:Path-decomposition}. The key observation, on which the proof strategy hinges on, is that Jensen's inequality applied on each interval $[\rho_{i},\rho_{i+1}]$ allows to prove there exists a ``gap'' between the two quantities. The main technical difficulty will be to control this ``gap'' uniformly in $i$. This control is established in Proposition \ref{prop:mainPropositon}.

In this section, unless specified otherwise, we assume that $X$ and $Y$ are two independent Ornstein-Uhlenbeck processes, with parameters $(\mu_1,\sigma_1)$ and $(\mu_2,\sigma_2)$ respectively, such that $Y_0 = 0$ and $X_0 = 1$. Before the main proof we present three technical lemmas. The first one is a concentration inequality for a conditioned Ornstein-Uhlenbeck process.
\begin{lem}
\label{lem:uniformTightness}Let $X$ be an Ornstein-Uhlenbeck process. For any $C_{1}>0$ there exist $C_{2},C_{3}>0$ such that for any $f:\R_{+}\mapsto\R_{+}$ being a $C_{1}$-Lipschitz function we have 
\begin{equation}
  \pr{X_{t}\geq x+f(t)|\forall_{s\leq t}X_{s}\geq f(s)}\leq\exp\rbr{-C_{2}x^{2}},\quad x\geq C_{3}f(t),
  \label{eq:concentrationAim-1}
\end{equation}
as soon as $X_{0}\in[f(0)+2,(C_{1}+1)f(0))$.
\end{lem}

\begin{proof}
To avoid cumbersome notation we assume that $t\in\mathbb{N}$. The proof for general $t$ follows similar lines. Further we assume that 
\begin{equation}
\forall_{s\leq t}f(s)\geq\min\cbr{(s+1)^{1/3},(t-s+1)^{1/3}},\label{eq:barrierAssumption}
\end{equation}
If it is not the case, by Lemma \ref{lem:StrongFKG} we can freely change $f$ by $s \mapsto f(s)+\min\cbr{(s+1)^{1/3},(t-s+1)^{1/3}}$ which is $(C_{1}+1)$-Lipschitz. 

We shorten $x_{t}\eqdef x+f(t)$ and let $c_{1}>1$. Using Lemma \ref{lem:StrongFKG} we estimate 
\begin{align}
  \pr{X_{t}\geq x_{t}|\forall_{s\leq t}X_{s}\geq f(s)} & \leq\pr{X_{t}\geq x_{t}\left|\forall_{s\leq t}X_{s}\geq f(s),\forall_{n\in\cbr{1,\ldots,t}}X_{n}\geq c_{1}f(n)\right.}\label{eq:tmp99-1}\\
  & =\frac{\pr{X_{t}\geq x_{t},\forall_{s\leq t}X_{s}\geq f(s)\left|\forall_{n\in\cbr{1,\ldots,t}}X_{n}\geq c_{1}f(n)\right.}}{\pr{\forall_{s\leq t}X_{s}\geq f(s)\left|\forall_{n\in\cbr{1,\ldots,t}}X_{n}\geq c_{1}f(n)\right.}}\nonumber \\
  & \leq\frac{\pr{X_{t}\geq x_{t}\left|\forall_{n\in\cbr{1,\ldots,t}}X_{n}\geq c_{1}f(n)\right.}}{\pr{\forall_{s\leq t}X_{s}\geq f(s)\left|\forall_{n\in\cbr{1,\ldots,t}}X_{n}\geq c_{1}f(n)\right.}}.\nonumber 
\end{align}
Let us first treat the denominator denoted by $I_{d}$. We use (\ref{eq:barrierAssumption}) and choose $c_{1}$ sufficiently large so that $I_{d}$ is bounded from below by a constant independent on $t$ and $f$. Using Lemma \ref{lem:StrongFKG} we obtain
\begin{align*}
  I_{d} & =\pr{\forall_{s\in[0,1]}X_{s}\geq f(s)\left|\forall_{n\in\cbr{1,\ldots,t}}X_{n}\geq c_{1}f(n)\right.}\\
  & \quad\times\pr{\forall_{s\in[1,t]}X_{s}\geq f(s)\left|{\forall_{n\in\cbr{1,\ldots,t}}X_{n}\geq c_{1}f(n)},{\forall_{s\in[0,1]}X_{s}\geq f(s)}\right.}\\
  & \geq\pr{\forall_{s\in[0,1]}X_{s}\geq f(s)\left|X_{1}\geq c_{1}f(1)\right.}\pr{\forall_{s\in[1,t]}X_{s}\geq f(s)\left|\forall_{n\in\cbr{1,\ldots,t}}X_{n}\geq c_{1}f(n)\right.}.
\end{align*}
Continuing in a same manner we obtain that
\begin{equation}
  I_{d}\geq\pr{\forall_{s\in[0,1]}X_{s}\geq f(s)|X_{1}\geq c_{1}f(1)}\times\prod_{n\in\cbr{1,\ldots,t}}\pr{\forall_{s\in[n-1,n]}X_{s}\geq f(s)\left|X_{n-1}=c_{1}f(n-1)\right.}.
  \label{eq:estimateHaHa}
\end{equation}
By point 3 of Fact \ref{fact:OUfact} conditionally on $X_{n-1} = c_1 f(n-1)$ the process $\cbr{\tilde{X}_t}_{t\in[0,1]}$ defined by $\tilde{X}_{t}\eqdef X_{n-1+t}-c_{1}e^{-\mu t}f(n-1)$ is an Ornstein-Uhlenbeck process starting from $\tilde{X}_{0}=0$. Thus
\begin{align*}
  \pr{\forall_{s\in[n-1,n]}X_{s}\geq f(s)\left|X_{n-1}=c_{1}f(n-1)\right.}  
  &\geq\pr{\forall_{s\in[0,1]}\tilde{X}_{s}\geq\sup_{s\in[n-1,n]}f(s)-c_{1}e^{-\mu}f(n-1)}\\
  &\geq\pr{\forall_{s\in[0,1]}\tilde{X}_{s}\geq-c_{1}e^{-\mu}f(n)/2}.
\end{align*}
We used inequality $\sup_{s\in[n-1,n]}f(s)-c_{1}e^{-\mu}f(n-1)\leq-c_{1}e^{-\mu}f(x)/2$ which can be easily verified by (\ref{eq:barrierAssumption}) and Lipschitz property as soon as $c_{1}$ is large enough. By point 4 of Fact \ref{fact:OUfact} we conclude that
\[
  \pr{\forall_{s\in[n-1,n]}X_{s}\geq f(s)\left|X_{n-1}=c_{1}f(n-1)\right.}\geq1-C\exp\rbr{-c\sbr{c_{1}e^{-\mu}f(n)/2}^{2}}.
\]
By this estimate, (\ref{eq:estimateHaHa}), (\ref{eq:barrierAssumption}) and increasing $c_{1}$ if necessary we obtain 
\[
  \pr{\forall_{s\leq t}X_{s}\geq f(s)\left|\forall_{n\leq t}X_{n-1}\geq c_{1}f(n)\right.}\geq p,
\]
for some $p>0$ which does not depend on $n$. From now on $c_{1}$ is fixed. Now in order to show (\ref{eq:concentrationAim-1}) it is enough to prove that the numerator in (\ref{eq:tmp99-1}) decays in a Gaussian fashion. This is the aim for the rest of the proof. We define a sequence $\cbr{G_{n}}_{n\geq0}$ by putting $G_{0}=X_{0}>0$ and 
\begin{equation}
  G_{n}\eqdef X_{n}-c_{n}X_{n-1},\quad n\geq1\label{eq:iidDecomposition}
\end{equation}
where $c_{n}\eqdef\frac{Cov(X_{n},X_{n-1})}{Cov(X_{n},X_{n})}.$ It is easy to check that in fact $c_{n}=c\in(0,1)$ and moreover the random variables $\cbr{G_{n}}_{n\geq0}$ are independent, distributed according to $\mathcal{N}(0,b^{2})$ where $b$ is a function of the parameters of the process $X$. We will prove that there exist $c_{2},C_{2}>0$ such that for any $x>C_{2}f(t)$ and $t\in\N$ we have
\begin{equation}
  \pr{X_{t}\geq x+f(t)\left|\forall_{n\in\cbr{1,\ldots,t}}X_{n}\geq c_{1}f(n)\right.}\leq e^{-c_{2}x^{2}}.
  \label{eq:inductionOUconcentration}
\end{equation}
We start by choosing constants $B,c_{2}>0$ satisfying
\begin{equation}
B\in(c,1),\quad c_{2}<\frac{(1-B)^{2}}{2b^{2}}.\label{eq:constantsAssumptions}
\end{equation}
Let $L\geq0$, without loss of generality we assume that $f(t)\geq L$. This assumption with the Lipschitz property yields that 
\begin{equation}
  A_{L}^{-1}\leq\frac{f(t+1)}{f(t)}\leq A_{L}.
  \label{eq:logLipshitz}
\end{equation}
for $A_{L}$ such that $A_{L}\searrow_{L}1$. We fix $L$ such that $B/(cA_{L})>1$. We proceed inductively. The constants $L$ and $C_{2}$ potentially may be increased during the further proof (the other constants stay fixed). We stress that this increase happens once and later the constants are valid for all steps of the induction. 

Checking the base case is an easy exercise. Let us assume that (\ref{eq:inductionOUconcentration}) holds for $t\geq0$. Let $x$ be such that $x+f(t+1)\geq c_{1}f(t+1)$, we have
\begin{align*}
  &\pr{X_{t+1}\geq x+f(t+1)\left|\forall_{n\in\cbr{1,\ldots,t+1}}X_{n}\geq c_{1}f(n)\right.}\\
   =&\frac{\pr{{X_{t+1}\geq x+f(t+1)},{\forall_{n\in\cbr{1,\ldots,t}}X_{n}\geq c_{1}f(n)}}}{\pr{\forall_{n\in\cbr{1,\ldots,t+1}}X_{n}\geq c_{1}f(n)}}\\
  =&\frac{\pr{X_{t+1}\geq x+f(t+1)\left|\forall_{n\in\cbr{1,\ldots,t}}X_{n}\geq c_{1}f(n)\right.}}{\pr{X_{t+1}\geq c_{1}f(t+1)\left|\forall_{n\in\cbr{1,\ldots,t}}X_{n}\geq c_{1}f(n)\right.}}.
\end{align*}
We denote the denominator by $I_{D}$. By (\ref{eq:iidDecomposition}) and (\ref{eq:logLipshitz}) we have
\begin{equation}
I_{D}\geq\pr{G_{t+1}\geq c_{1}f(t+1)-cc_{1}f(t)}\geq\pr{G_{t+1}\geq c_{1}(A_{L}-c)f(t)}.\label{eq:id}
\end{equation}
By (\ref{eq:iidDecomposition}) and the union bound we conclude that the numerator is smaller than $I_{n}^{1}+I_{n}^{2}$, where 
\[
  I_{n}^{1}\eqdef\pr{X_{t}\geq\frac{B}{c}\rbr{x+f(t+1)}\left|\forall_{n\in\cbr{1,\ldots,t}}X_{n}\geq c_{1}f(n)\right.},\quad I_{n}^{2}\eqdef\pr{G_{t+1}\geq(1-B)\rbr{x+f(t+1)}}.
\]
Let $x>C_{2}f(t+1)$ then 
\[
\frac{B}{c}\rbr{x+f(t+1)}-c_{1}f(t)\geq\frac{B}{c}(C_{2}+1)f(t+1)-c_{1}f(t)\geq f(t)\sbr{\frac{B}{cA_{L}}\rbr{C_{2}+1}-c_{1}}\geq C_{2}f(t).
\]
We assumed that $B/(cA_{L})>1$ thus the last inequality holds if we choose $C_{2}$ large enough. We can thus use the induction hypothesis (\ref{eq:inductionOUconcentration}) for $t$. We have
\begin{align*}
I_n^1  \leq\exp\cbr{-c_{2}\rbr{\frac{B}{c}x+\frac{B}{c}f(t+1)-f(t)}^{2}} \leq\exp\cbr{-c_{2}\rbr{\frac{B}{c}x+[B/(cA_{L})-1]f(t)}^{2}}.
\end{align*}
Recalling (\ref{eq:id}) and increasing $L$ so that $c_{1}(A_{L}-c)f(t)\geq c_{1}(A_{L}-c)L\geq2$ holds we can use the Gaussian tail estimate (\ref{eq:gaussianTail}) as follows 
\begin{align*}
\frac{I_{n}^{1}}{I_{D}} & \leq\frac{5c_{1}(A_{L}-c)f(t)}{b}\exp\cbr{\frac{c_{1}^{2}(A_{L}-c)^{2}f(t)^{2}}{2b^{2}}}\exp\cbr{-c_{2}\rbr{\frac{B}{c}x+[B/(cA_{L})-1]f(t)}^{2}}\\
 & \leq\frac{5c_{1}(A_{L}-c)f(t)}{b}\exp\cbr{\frac{c_{1}^{2}(A_{L}-c)^{2}}{2b^{2}}f(t)^{2}-\frac{2c_{2}B[B/(cA_{L})-1]}{c}xf(t)}\exp\cbr{-\frac{c_{2}B^{2}}{c^{2}}x^{2}}.
\end{align*}
We increase $C_{2}$ (we recall that $x\geq C_{2}f(t+1)$) so that $\frac{c_{1}^{2}(A_{L}-c)^{2}}{2b^{2}}<C_{2}\frac{2c_{2}B[B/(cA_{L})-1]}{cA_{L}}$. Then we increase $L$ if necessary so that the first two factors are bounded by $1/2$. Finally 
\[
\frac{I_{n}^{1}}{I_{D}}\leq\frac{1}{2}\exp\cbr{-\frac{c_{2}B^{2}}{c^{2}}x^{2}},
\]
which by (\ref{eq:constantsAssumptions}) implies $I_{n}^{1}/I_{D}\leq\exp\cbr{-c_{2}x^{2}}/2$. We perform similar calculations for $I_{n}^{2}$: 
\begin{align*}
\frac{I_{n}^{2}}{I_{D}} & \leq\frac{5c_{1}(A_{L}-c)f(t)}{b}\exp\cbr{\frac{c_{1}^{2}(A_{L}-c)^{2}f(t)^{2}}{2b^{2}}}\exp\cbr{-\frac{(1-B)^{2}}{2b^{2}}\rbr{x+f(t)/A_{L}}^{2}}\\
 & \leq\frac{5c_{1}(A_{L}-c)f(t)}{b}\exp\cbr{\frac{c_{1}^{2}(A_{L}-c)^{2}}{2b^{2}}f(t)^{2}-\frac{(1-B)^{2}}{b^{2}A_{L}}xf(t)}\exp\cbr{-\frac{(1-B)^{2}}{2b^{2}}x^{2}}\\
 & \leq\frac{1}{2}\exp\cbr{-\frac{(1-B)^{2}}{2b^{2}}x^{2}},
\end{align*}
where the last estimates follows by increasing $C_{2}$ and $L$ if necessary (analogously to the previous case). Now, by (\ref{eq:constantsAssumptions}) follows $I_{n}^{2}/I_{D}\leq\exp(-c_{2}x^{2})/2$. Recalling the previous step we obtain $(I_{n}^{1}+I_{n}^{2})/I_{D}\leq\exp(-c_{2}x^{2})$ which establish (\ref{eq:inductionOUconcentration}) for $t+1$.
\end{proof}

A similar property holds for conditioning in future.
\begin{lem}
\label{lem:forwardConditioningConcerntration}
Let $X$ be an Ornstein-Uhlenbeck process  and $u\in[c,C]$, for $C>c>0$. Then there exist $C_{1},c_{1}>0$ such that for any $t>1$ we have
\[
  \pr{X_{u}\geq x|\forall_{s\in[u,u+t]}X_{s}\geq(1+s-u)^{1/3}}\leq C_{1}e^{-c_{1}x^{2}}.
\]
\end{lem}

\begin{proof}
We set $f(s)\eqdef(1+\lceil s\rceil)^{1/3}$, by Lemma \ref{lem:StrongFKG} it is enough show the claim with $f(s)$ instead of $(1+s)^{1/3}$. Using the Markov property we write 
\begin{align*}
  \pr{X_{u}\geq x|\forall_{s\in[u,u+t]}X_{s}>f(s-u)} & =\frac{\pr{\cbr{X_{u}\geq x}\cap\cbr{\forall_{s\in[u,u+t]}X_{s}>f(s-u)}}}{\pr{\forall_{s\in[u,u+t]}X_{s}>f(s-u)}}\\
  & =\frac{\int_{x}^{+\infty}w(y,t)\pr{X_{u}\in\dd y}}{\int_{0}^{+\infty}w(y,t)\pr{X_{u}\in\dd y}},
\end{align*}
where
\[
  w(y,t)=\pr{\forall_{s\leq t}X_{s}\geq f(s)|X_{0}=y}.
\]
The function is increasing with respect to $y$, thus the denominator can be estimated as
\[
	\int_{3}^{4}w(y,t)\pr{X_{u}\in\dd y} \geq w(3,t)\cdot \pp(X_u \in [3,4]).
\]
By the Gaussian concentration of $X_{u}$, one checks that to show the claim it is enough to that 
\begin{equation}
  \frac{w(x,t)}{w(3,t)}\leq \exp\rbr{C_{1}(\log x)^{5/3}},\label{eq:ulala}
\end{equation}
for $C_{1}>0$ for $x\geq3$. It will be easier to rewrite $w$ as $w(x,t)=\pr{\forall_{s\leq t}X_{s}\geq f(s)-xe^{-\mu s}}$ with the assumption that $X_{0}=0$. Let us set $t_{x}\eqdef\lceil C_{t}\log x\rceil$, where $C_{t}>0$ will be adjusted later. For $x>3$ we have
\begin{align*}
  w(x,t) & =\pr{\forall_{s\in[t_{x},t]}X_{s}\geq f(s)-xe^{-\mu s}|\forall_{s\leq t_{x}}X_{s}\geq f(s)-xe^{-\mu s}}\pr{\forall_{s\leq t_{x}}X_{s}\geq f(s)-xe^{-\mu s}}\\
  &\leq\pr{\forall_{s\in[t_{x},t]}X_{s}\geq f(s)-xe^{-\mu s}|\forall_{s\leq t_{x}}X_{s}\geq f(s)-xe^{-\mu s}}\\
  &\leq\pr{\forall_{s\in[t_{x},t]}X_{s}\geq f(s)-xe^{-\mu s}|\forall_{s\leq t_{x}}X_{s}\geq f(s)+1-3e^{-\mu s}},
\end{align*}
where in the last line we used Lemma \ref{lem:StrongFKG}. Moreover, by convention we assume that the probability above is $1$ if $t_{x}\geq t$. Similarly we estimate
\begin{align*}
  w(3,t) & \geq\pr{{\forall_{s\leq t_{x}}X_{s}\geq f(s)+1-3e^{-\mu s}},{\forall_{s\in[t_{x},t]}X_{s}\geq f(s)}}\\
  & \geq\pr{\forall_{s\leq t_{x}}X_{s}\geq f(s)+1-3e^{-\mu s}}\pr{\forall_{s\in[t_{x},t]}X_{s}\geq f(s)|\forall_{s\leq t_{x}}X_{s}\geq f(s)+1-3e^{-\mu s}}.
\end{align*}
Using calculations similar to (\ref{eq:estimateHaHa}) and $f(t_{x})=O((\log x)^{1/3})$ one can show that 
\[
  \pr{\forall_{s\leq t_{x}}X_{s}\geq f(s)+1-3e^{-\mu s}}\geq c_{3}e^{-C_{3}(\log x)^{2/3}t_{x}},
\]
for some $c_{3},C_{3}>0$. Now we will show that for $C_{t}$ large enough (recall that $t_{x}\eqdef\lceil C_{t}\log x\rceil$) and $y\geq f(t_{x})+1-3e^{-\mu t_{x}}$ there exists a constant $c>0$ such that for any $t > t_x$ we have
\begin{multline*}
  \frac{\pr{\forall_{s\in[t_{x},t]}X_{s}\geq f(s)|X_{t_{x}}=y}}{\pr{\forall_{s\in[t_{x},t]}X_{s}\geq f(s)-xe^{-\mu s}|X_{t_{x}}=y}}\\
  =\pr{\forall_{s\in[t_{x},t]}X_{s}\geq f(s)|\forall_{s\in[t_{x},t]}X_{s}\geq f(s)-xe^{-\mu s},X_{t_{x}}=y}\geq c.
\end{multline*}
Integrating one verifies that
\[
	\frac{\pr{\forall_{s\in[t_{x},t]}X_{s}\geq f(s)-xe^{-\mu s}|\forall_{s\leq t_{x}}X_{s}\geq f(s)+1-3e^{-\mu s}}}{\pr{\forall_{s\in[t_{x},t]}X_{s}\geq f(s)|\forall_{s\leq t_{x}}X_{s}\geq f(s)+1-3e^{-\mu s}}} \leq c^{-1},
\]
which is enough to conclude the proof of (\ref{eq:ulala}) and consequently the proof of the lemma. Equivalently we will show that 
\begin{equation}
H\eqdef\pr{\exists_{s\in[t_{x},t]}X_{s}\leq f(s)|\forall_{s\in[t_{x},t]}X_{s}\geq f(s)-xe^{-\mu s},X_{t_{x}}=y}\leq1-c.\label{eq:HEstimate}
\end{equation}
We consider 
\begin{align*}
 H& \leq\sum_{k=t_{x}}^{\lceil t\rceil}\pr{\exists_{s\in[k,k+1)}X_{s}\leq f(s)|\forall_{s\leq t}X_{s}\geq f(s)-xe^{-\mu s},X_{t_{x}}=y}\\
 & \leq\sum_{k=t_{x}}^{\lceil t\rceil}\pr{\exists_{s\in[k,k+1)}X_{s}\leq f(k)|\forall_{s\in[k,k+1]}X_{s}\geq f(k)-xe^{-\mu s},X_{t_{x}}=y}.
\end{align*}
The first inequality follows by the union bound and the second one by the assumption on $f$ and Lemma~\ref{lem:StrongFKG}. The first term (i.e. $k=t_{x}$) can be made arbitrarily small by choosing $C_{t}$ (and thus $t_{x}$) large. To estimate the other terms we define a function $p_{k}:\R_{+}\mapsto\R$ by 
\[
  p_{k}(A)=-\log\pr{\forall_{s\in[k,k+1]}X_{s}\geq A|X_{t_{x}}=y}.
\]
By Lemma \ref{lem:covexTail} one deduces that $p$ is convex. Further we notice 
\[
  \pr{\forall_{s\in[k,k+1]}X_{s}\geq A|X_{t_{x}}=y}\geq\pr{X_{k}\geq Ae^{\mu}|X_{t_{x}}=y}\pr{\forall_{s\in[k,k+1]}X_{s}\geq A|X_{k}=e^{\mu}A}.
\]
The second factor can be easily bounded from below by a strictly positive constant uniform in $A,k$. Thus for some $C_{4}>0$ we have $p_{k}(A)\leq C_{4}(A+1)^{2}$. Using the convexity of $p_{k}$ it is easy to deduce that for some $C_{5}>0$ we have $p'_{k}(A)\leq C_{5}(A+1),$  where $p'_{k}$ denotes the left derivative of $p_{k}$. Thus 
\begin{multline*}
  \log\pr{\forall_{s\in[k,k+1]}X_{t}\geq f(k)|\forall_{s\in[k,k+1]}X_{t}\geq f(k)-xe^{-\mu k}}\\
  =p_{k}(f(k)-xe^{-\mu k})-p_{k}(f(k))\geq-C_{5}xe^{-\mu k}f(k).
\end{multline*}
Now we can make the final estimate. We write 
\[
\sum_{k=t_{x}+1}^{\lceil t\rceil}\rbr{1-\exp(-C_{5}xe^{-\mu k}f(k))}\leq C_{6}\sum_{k=\lfloor t_{x}\rfloor}^{\lceil t\rceil}f(k)e^{-\mu(k-\lceil\log k/\mu\rceil)}.
\]
for some $C_{6}>0$. Increasing $C_{t}$ (recall that $t_{x}\eqdef\lceil C_{t}\log x\rceil$) if necessary, we can make the sum arbitrarily small, proving (\ref{eq:HEstimate}) and concluding the proof.
\end{proof}

We introduce $b_{t}:[0,t]\mapsto\R_{+}$ by 
\begin{equation}
b_{t}(s)\eqdef\min\cbr{(s+1)^{1/3},(t-s+1)^{1/3}}.\label{eq:barrier}
\end{equation}
Let us recall the notation of Section \ref{sub:Path-decomposition}. We have
\begin{lem}
\label{lem:boundedOU}
There exist $C,c>0$ such that for any $n,k\in\mathbb{N}$ we have 
\begin{equation}
\pr{\forall_{s\leq\rho_{k}}Y_{s}\leq Cb_{\rho_{k}}(s)}>c,\quad\pr{\cbr{\pr{\mathcal{B}_{n,k}|\cbr{r_{i}}}\geq1/10}}>c,\label{eq:tentBound}
\end{equation}
where $\mathcal{B}_{n,k}\eqdef\cbr{\forall_{\rho_{k+1}\leq s\leq\rho_{n}}Y_{s}\leq C(s-\rho_{k+1}+1)^{1/3}}$.
\end{lem}

\begin{proof}
The proof of this result is rather standard, thus we only present a sketch of it. We observe that for any $k \in \N$, we have
\begin{equation}
  \label{eqn:stepone}
  \pr{\forall_{s\leq\rho_{k}}Y_{s}\leq Cb_{\rho_{k}}(s)} \leq 1 - \pr{\exists_{s \geq 0} : Y_s \geq C (1+s)^{1/3}} - \pr{\exists_{s \in [0,\rho_k]} : Y_s \geq C (1 + (\rho_k-s))^{1/3}}.
\end{equation}
We note that by \eqref{eq:OUandBMRelation} and the law of iterated logarithm, we have
\[
  \sup_{s \geq 0} \frac{Y_s}{(1+s)^{1/3}} < +\infty \quad \text{a.s.}
\]
thus the first term in \eqref{eqn:stepone} can be made as small as wished by choosing $C$ large enough. To treat the second term, we use the decomposition in excursions of the Ornstein-Uhlenbeck process and observe that $(Y_{\rho_k-s})$ has the same law as the concatenation of an excursion conditioned to be larger than $1$ and an independent Ornstein-Uhlenbeck process. Thus, using William's decomposition and  the law of iterated logarithm again, for a given $\epsilon>0$ we can choose $C>0$ large enough such that
\[
  \pr{\exists_{s \in [0,\rho_k]} : Y_s \geq C (1 + (\rho_k-s))^{1/3}} < \epsilon.
\]

We now prove the second inequality. By the Markov inequality, we have
\[
  \pr{\pr{ \mathcal{B}_{n,k}^c| \{r_i\}}>9/10} \leq \frac{10}{9} \pr{\mathcal{B}^c_{n,k}}.
\]
Using now the Markov property at time $\rho_k$, we have
\[
  \pr{\mathcal{B}_{n,k}^c} \leq \pr{\exists_{s \geq 0} : Y_s \geq C (1-s)^{1/3})},
\]
therefore choosing $C$ large enough, we can make this probability as small as wished.
\end{proof}

\begin{proof}
The proof is rather standard therefore we present only a sketch. We denote $f(s)\eqdef C(s+1)^{1/3}$. We consider
\begin{align}
  \pr{\forall_{s\leq\rho_{k}}Y_{s}\leq Cb_{\rho_{k}}(s)} & =1-\pr{\exists_{s\leq\rho_{k}}Y_{s}\geq Cb_{\rho_{k}}(s)}\label{eq:uLala}\\
  &\geq1-\pr{\exists_{s\leq\rho_{k}}Y_{s}\geq Cf(s)}-\pr{\exists_{s\leq\rho_{k}}Y_{s}\geq Cf(\rho_{k}-s)}.\nonumber 
\end{align}
Let us treat the second term. Let $l\in\mathbb{N}$, we have
\begin{equation}
  \pr{\exists_{s\leq\rho_{k}}Y_{s}\geq Cf(s)}
  \leq\pr{\exists_{s\geq0}Y_{s}\geq Cf(s)}
  \leq\pr{\exists_{s\in[0,\rho_{l}]}Y_{s}\geq Cf(s)}+\sum_{i=l}^{+\infty}\pr{\exists_{s\in[\rho_{i},\rho_{i+1}]}Y_{s}\geq Cf(s)}.
  \label{eq:middleStep}
\end{equation}
We recall Fact \ref{fact:decompositon} and the notation there. For large enough $i$ and some $c>0$ we have
\begin{align*}
  \pr{\exists_{s\in[\rho_{i},\rho_{i+1}]}Y_{s}\geq Cf(s)}
  \leq & \pr{\cbr{M^{i}\geq Cf(\rho_{i})}\cap\cbr{\rho_{i}\geq ci}}+\pr{\rho_{i}<ci}\\
  \leq & \pr{M^{i}\geq Cf(ci)}+\pr{\rho_{i}<ci}\\
  \leq & e^{-C_{1}Cf(ci)}+e^{-C_{2}i},
\end{align*}
where $C_{1},C_{2}>0$. Increasing $l$ and $C$ one can make (\ref{eq:middleStep}) as small as we want. Treating the third term of (\ref{eq:uLala}) similarly we obtain the first statement of (\ref{eq:tentBound}). We set $\mathcal{A}_{n,k}\eqdef\cbr{\pr{\mathcal{B}_{n,k}|\cbr{r_{i}}}\geq1/10}$ and $p\eqdef\pr{\mathcal{A}_{n,k}}.$ We have
\begin{align*}
  \pr{\mathcal{B}_{n,k}} & =\E\E[1_{\mathcal{B}_{n,k}}|\cbr{r_{i}}]=\E\sbr{1_{\mathcal{A}_{n,k}}\E[1_{\mathcal{B}_{n,k}}|\cbr{r_{i}}]}+\E\sbr{1_{\mathcal{A}_{n,k}^{c}}\E[1_{\mathcal{B}_{n,k}}|\cbr{r_{i}}]}\\
  & \leq p+(1-p)/10=\frac{1}{10}+\frac{9}{10}p.
\end{align*}
By the first argument we choose $C$ such that $\pr{\mathcal{B}_{n,k}}>1/10$ which implies $p>0$ (uniformly in $n$ and $k$).
\end{proof}

\subsection{Reformulation of the problem}
\label{sub:Reformulation-of-the}
We introduce necessary notions and reformulate the problem. Let $\meas$ be the space of finite measures on $\R_{+}$. Given $m\in\meas$ we denote $\norm m{}{}\eqdef\int_{\R_{+}}m(\dd x)$. Let $\mathcal{\pspace}$ be the functional space
\begin{equation}
  \mathcal{P}\eqdef\cbr{f:f\text{ is a continuous function from an interval }[0,t]\text{ to }\R\mbox{\text{ such that }}f(0)=f(t)=0}.\label{eq:pspace}
\end{equation}
Let us define an operator $T:\mathcal{M\times\mathcal{P}}\mapsto\mathcal{M}$. Given a measure $m\in\mathcal{M}$ and $f\in\mathcal{P}$ such that $f:[0,t]\mapsto\R$ we set $T(m,f)\eqdef\tilde{m}$ defined by 
\[
  \tilde{m}(\dd x)\eqdef\norm m{}{}\mathbb{P}_{m/\norm m{}{}}\rbr{\forall_{s\leq t}X_{s}\geq f(s),X_{t}\in\dd x},
\]
where under $\mathbb{P}_{m/\norm m{}{}}$ the process $X$ is an Ornstein-Uhlenbeck process such that $X_{0}=^{d}m/\norm m{}{}$.

For $n\in\N$ we define iteratively $\mathcal{M}$-valued random variables $T_{n}$ by 
\begin{equation}
  T_{n}\eqdef
  \begin{cases}
    \delta_{1} & n=0\\
    T(T_{n-1},Y^{n-1}) & n>0
  \end{cases},
  \label{eq:defTn}
\end{equation}
where $Y^{n}$ is given by (\ref{eq:Yi}). Using the Markov property one proves by induction that
\begin{equation}
  T_{n}=\pr{\forall_{s\leq\rho_{n}}X_{s}\geq Y_{s},X_{\rho_{n}}\in\dd x|Y}.
  \label{eq:tmp99}
\end{equation}
Let us denote $\mathcal{F}\eqdef\sigma(\rho_{i},i\in\N)$. The following lemma relates $T_{n}$ to our original problem.
\begin{lem}
\label{lem:relation}Let $\gamma_{\mu_{1},\mu_{2}}$ and $\delta_{\mu_{1},\mu_{2}}$
be the same as in Theorem \ref{thm:OUresult}. Then 
\begin{equation}
  \gamma_{\mu_{1},\mu_{2}}=\lim_{n\to+\infty}\frac{\ev{}\sbr{-\log\norm{T_{n}}{}{}}}{n\ev{}r_{1}},
  \quad\delta_{\mu_{1},\mu_{2}}\leq\liminf_{n\to+\infty}\frac{\ev{}\sbr{-\log\ev{}\rbr{\norm{T_{n}}{}{}|\mathcal{F}}}}{n\ev{}r_{1}}.
  \label{eq:relationLemmaClaim}
\end{equation}
\end{lem}

The following proposition is the main technical result of this proof 
\begin{prop}
\label{prop:mainPropositon}There exist $c>0$ and $n_{0}\in\mathbb{N}$
such that 
\begin{equation}
\ev{}\log\norm{T_{n}}{}{}-\ev{}\log\ev{}(\norm{T_{n}}{}{}|\mathcal{F})\leq-nc,\label{eq:smallAim}
\end{equation}
for any $n\geq n_{0}$.
\end{prop}
We observe that this proposition together with Lemma \ref{lem:relation} imply (\ref{eq:mainAim}).

\begin{proof}[Proof of Lemma \ref{lem:relation}] The first convergence in (\ref{eq:relationLemmaClaim}) holds by \eqref{eq:limitAll2} (recall also relation between $\tilde{\gamma}$ and $\gamma_{\mu_{1},\mu_{2}}$ given in (\ref{eq:xyz-1})). We observe that (\ref{eq:tmp99}) yields $\ev{}\rbr{\norm{T_{n}}{}{}|\mathcal{F}}=\pr{\forall_{s\leq\rho_{n}}X_{s}\geq Y_{s}|\mathcal{F}}$. We note that methods of Section \ref{sub:Proof-of-convergence} imply that 
\begin{equation}
  \frac{-\log\pr{\forall_{s\leq\rho_{n}}X_{s}\geq Y_{s}|\mathcal{F}}}{n}
  \label{eq:tmporaryConvergence}
\end{equation}
converges a.s. and in $L^{1}$ for the sake of brevity we skip details. We define $r(n)\eqdef\ev{}\rho_{n}-n^{2/3}=n\ev{}r_{1}-n^{2/3}$ and a sequence of events $\mathcal{A}_{n}\eqdef\cbr{\rho_{n}\geq r(n)}$. Using Fact \ref{fact:decompositon} one proves $1_{\mathcal{A}_{n}^{c}}\to0$ a.s. Consequently the convergence of (\ref{eq:tmporaryConvergence}) implies 
\[
  \lim_{n\to+\infty}\frac{\ev{}1_{\mathcal{A}_{n}^{c}}\log\pr{\forall_{s\leq r(n)}X_{s}\geq Y_{s}|\mathcal{F}}}{n}=0.
\]
Using $\ev{}\rbr{\norm{T_{n}}{}{}|\mathcal{F}}\leq1$ we estimate
\begin{align}
  \ev{}\log\ev{}\rbr{\norm{T_{n}}{}{}|\mathcal{F}}
  &\leq\ev{}1_{\mathcal{A}_{n}}\log\ev{}\rbr{\norm{T_{n}}{}{}|\mathcal{F}}\nonumber\\
  &\leq\ev{}1_{\mathcal{A}_{n}}\log\pr{\forall_{s\leq r(n)}X_{s}\geq Y_{s}|\mathcal{F}}\nonumber \\
  &=\ev{}\log\pr{\forall_{s\leq r(n)}X_{s}\geq Y_{s}|\mathcal{F}}-\ev{}1_{\mathcal{A}_{n}^{c}}\log\pr{\forall_{s\leq r(n)}X_{s}\geq Y_{s}|\mathcal{F}}.\label{eq:tmp19} 
\end{align}
We denote the first term of the right-hand side of (\ref{eq:tmp19}) by $J_{n}$. Applying Jensen's inequality we get 
\[
  J_{n}\leq\log\pr{\forall_{s\leq r(n)}X_{s}\geq Y_{s}}.
\]
By (\ref{eq:deltaDef}) and the definition of $r(n)$ we have $\limsup_{n}J_{n}/(n\ev{}r_{1})\leq-\delta_{\mu_{1},\mu_{2}}$ and the second term (\ref{eq:tmp19}) can be shown to converge to $0$. We conclude that the second claim of (\ref{eq:relationLemmaClaim}) holds.
\end{proof}


\subsection{Proof of Proposition \ref{prop:mainPropositon}}

\begin{proof}[Proof of Proposition \ref{prop:mainPropositon}]
We recall $\mathcal{F}=\sigma(\rho_{i},i\in\N)$ and define a filtration
$\cbr{\mathcal{F}_{k}}_{k\geq0}$ by putting $\mathcal{F}_{0}\eqdef\cbr{\emptyset,\Omega}$
and 
\[
\mathcal{F}_{k}\eqdef\sigma\cbr{Y^{i}:i<k},\quad k>0,
\]
(see also Figure \ref{fig:notation}). We recall (\ref{eq:defTn})
and for $k\in\cbr{0,1,\ldots,n}$ define $\meas$-valued random variables
$T_{n}^{k}$ by 
\[
T_{n}^{k}\eqdef\ev{}(T_{n}|\mathcal{F}_{k},\mathcal{F}).
\]
This definition and (\ref{eq:tmp99}) imply that 
\begin{equation}
T_{n}^{k}=\mathbb{P}\rbr{\forall_{s\leq\rho_{n}}X_{s}\geq Y_{s},X_{\rho_{n}}\in\dd x|\cbr{Y^{i}}_{i<k},\mathcal{F}},\quad T_{n}^{n}=T_{n}.\label{eq:tnk}
\end{equation}
By the Markov property of $X$ we have
\begin{equation}
\log\norm{T_{n}^{k+1}}{}{}=\log\norm{T_{k}}{}{}+\log\mathbb{P}_{T_{k}/\norm{T_{k}}{}{}}\rbr{\forall_{\rho_{k}\leq s\leq\rho_{n}}X_{s}\geq Y_{s}|Y^{k},\mathcal{F}}.\label{eq:statartingDecomposition}
\end{equation}
This expression requires some comment. We recall that $T_{k}$ is
a random measure, conditionally on $m=T_{k}/\norm{T_{k}}{}{}$ we
understand $X$ to be an Ornstein-Uhlenbeck process starting from
$m$ at time $\rho_{k}$. Let us now denote 
\[
G_{n,k}\eqdef\E\sbr{\log\mathbb{P}_{T_{k}/\norm{T_{k}}{}{}}\rbr{\forall_{\rho_{k}\leq s\leq\rho_{n}}X_{s}\geq Y_{s}|Y^{k},\mathcal{F}}|\mathcal{F}_{k},\mathcal{F}}-\log\mathbb{P}_{T_{k}/\norm{T_{k}}{}{}}\rbr{\forall_{\rho_{k}\leq s\leq\rho_{n}}X_{s}\geq Y_{s}|\mathcal{F}}.
\]
We notice that $G_{n,k}$ is a random variable, which by Jensen's
inequality fulfills $G_{n,k}\leq0$ (we will prove strict inequality
later). In this notation (\ref{eq:statartingDecomposition}) yields
\[
\E\sbr{\log\norm{T_{n}^{k+1}}{}{}|\mathcal{F}_{k},\mathcal{F}}=\log\norm{T_{k}}{}{}+\log\mathbb{P}_{T_{k}/\norm{T_{k}}{}{}}\rbr{\forall_{\rho_{k}\leq s\leq\rho_{n}}X_{s}\geq Y_{s}|\mathcal{F}}+G_{n,k}=\log\norm{T_{n}^{k}}{}{}+G_{n,k}.
\]
We apply this relation iteratively
\begin{align*}
\ev{}\sbr{\log\norm{T_{n}}{}{}|\mathcal{F}}=\ev{}\sbr{\log\norm{T_{n}^{n}}{}{}|\mathcal{F}} & =\ev{}\sbr{\ev{}\sbr{\log\norm{T_{n}^{n}}{}{}|\mathcal{F}_{n-1},\mathcal{F}}|\mathcal{F}}\\
 & =\ev{}\sbr{\log\norm{T_{n}^{n-1}}{}{}|\mathcal{F}}+\ev{}[G_{n,n-1}|\mathcal{F}]\\
 & =\ev{}\sbr{\ev{}\sbr{\log\norm{T_{n}^{n-1}}{}{}|\mathcal{F}_{n-2},\mathcal{F}}}+\ev{}[G_{n,n-1}|\mathcal{F}]\\
 & =\ev{}\sbr{\log\norm{T_{n}^{n-2}}{}{}|\mathcal{F}}+\ev{}[G_{n,n-2}|\mathcal{F}]+\ev{}[G_{n,n-1}|\mathcal{F}]\\
 & =\ldots\\
 & =\ev{}\sbr{\log\norm{T_{n}^{0}}{}{}|\mathcal{F}}+\sum_{k=0}^{n-1}\ev{}[G_{n,k}|\mathcal{F}].\\
\end{align*}
We notice that $\norm{T_{n}^{0}}{}{}=\mathbb{P}\rbr{\forall_{s\leq\rho_{n}}X_{s}\geq Y_{s}|\mathcal{F}}=\E\sbr{\norm{T_{n}}{}{}|\mathcal{F}}$.
Thus 
\[
\ev{}\sbr{\log\norm{T_{n}}{}{}}=\E\log\ev{}\sbr{\norm{T_{n}}{}{}|\mathcal{F}}+\sum_{k=0}^{n-1}\E G_{n,k}.
\]
One easily sees that an inequality 
\begin{equation}
\E G_{n,k}\leq c,\label{eq:gapGoal}
\end{equation}
for some $c<0$, is sufficient to conclude the proof of the proposition.
Proving (\ref{eq:gapGoal}) is our aim now. To avoid heavy notation
we denote $\tilde{\E}(\cdot)\eqdef\E(\cdot|\mathcal{F})$ and $\tilde{\E}_{k}(\cdot)\eqdef\tilde{\E}(\cdot|\mathcal{F}_{k})$.
\textcolor{blue}{}\global\long\def\ev#1{\tilde{\mathbb{E}}{#1}}
\textcolor{blue}{}\global\long\def\E{\tilde{\mathbb{E}}}
\textcolor{blue}{}\global\long\def\pp{\tilde{\mathbb{P}}}
\textcolor{blue}{}\global\long\def\pr#1{\tilde{\mathbb{P}}\rbr{#1}}

Further, we introduce additional randomization: a probability measure
$\mathbb{P}_{\pm}$ and the random variable $\eta$ such that $\mathbb{P}_{\pm}(\eta=1)=\mathbb{P}_{\pm}(\eta=-1)=1/2$.
For $k\in\mathbb{N}$ we define $\cbr{\tilde{Y}_{s}^{k}(\eta)}_{s\geq0}$
by 
\[
\tilde{Y}_{s}^{k}(\eta)\eqdef\begin{cases}
\eta|Y_{s}| & \text{if }s\in[\tau_{k},\rho_{k+1}]\\
Y_{s} & \text{otherwise.}
\end{cases}.
\]
There are two easy but crucial observations to be made at this point.
Firstly,
\begin{equation}
\forall_{s\geq0}\tilde{Y}_{s}^{k}(1)\geq\tilde{Y}_{s}^{k}(-1)\text{ and }\forall_{s\in(\tau_{k},\rho_{k+1})}\tilde{Y}_{s}^{k}(1)>\tilde{Y}_{s}^{k}(-1).\label{eq:ytildeDominance}
\end{equation}
Secondly, the excursions of an Ornstein-Uhlenbeck process are symmetric
around $0$. Formally, under $\mathbb{E}_{\pm}\otimes\E_{k}$ the
process $\tilde{Y}^{k}(\eta)$ has the same law as $Y$ under $\E_{k}$.
Let us shorten $m\eqdef T_{k}/\norm{T_{k}}{}{}$ and denote ``the
gap''
\begin{equation}
\Delta_{n,k}\eqdef\mathbb{E}_{\pm}\E_{k}\log\pp_{m}\rbr{\forall_{\rho_{k}\leq s\leq\rho_{n}}X_{s}\geq\tilde{Y}_{s}^{k}(\eta)|Y^{k}}-\E_{k}\log\mathbb{E}_{\pm}\pp_{m}\rbr{\forall_{\rho_{k}\leq s\leq\rho_{n}}X_{s}\geq\tilde{Y}_{s}^{k}(\eta)|Y^{k}}.\label{eq:mindTheGap}
\end{equation}
By Jensen's inequality we have $G_{n,k}\leq\Delta_{n,k}\leq0$. In
order to show (\ref{eq:gapGoal}) we will obtain a bound from above
on $\Delta_{n,k}$ which is strictly negative and uniform in $n,k$.
We define
\begin{equation}
g_{n,k}\eqdef\frac{\pp_{m}\rbr{\forall_{\rho_{k}\leq s\leq\rho_{n}}X_{s}\geq\tilde{Y}_{s}^{k}(1)|Y^{k}}}{\pp_{m}\rbr{\forall_{\rho_{k}\leq s\leq\rho_{n}}X_{s}\geq\tilde{Y}_{s}^{k}(-1)|Y^{k}}},\label{eq:gnk}
\end{equation}
and $z_{n,k}\eqdef\pp_{m}\rbr{\forall_{\rho_{k}\leq s\leq\rho_{n}}X_{s}\geq\tilde{Y}_{s}^{k}(-1)|Y^{k}}$.
In this notation (\ref{eq:mindTheGap}) writes as 
\begin{align}
\Delta_{n,k} & =\E_{k}\sbr{\frac{1}{2}\rbr{\log(g_{n,k}z_{n,k})+\log z_{n,k}}-\log\rbr{\frac{g_{n,k}z_{n,k}+z_{n,k}}{2}}}=\E_{k}\sbr{\frac{1}{2}\log g_{n,k}-\log\rbr{\frac{g_{n,k}+1}{2}}}\label{eq:DeltaGapEstimate}\\
 & \leq-\frac{1}{8}\E_{k}(g_{n,k}-1)^{2}.\nonumber 
\end{align}
To explain the last inequality we observe that (\ref{eq:ytildeDominance})
yields $g_{n,k}\leq1$ and that for $x\in(0,1]$ we have and elementary
inequality $\frac{1}{2}\log x-\log\rbr{\frac{x+1}{2}}\leq-\frac{1}{8}(x-1)^{2}.$
Now we concentrate on proving that in fact, uniformly in $n,k$ we
have $g_{n,k}<1$. Let us analyze the expressions appearing in (\ref{eq:gnk}).
We denote $\mathbb{Q}_{m,k}(\cdot)\eqdef\mathbb{P}(\cdot|Y^{k})$,
$\mathcal{A}_{i}\eqdef\cbr{\forall_{\rho_{k}\leq s\leq\rho_{k+1}}X_{s}\geq\tilde{Y}_{s}^{k}(i)}$
and $\mathcal{B}\eqdef\cbr{\forall_{\rho_{k+1}\leq s\leq\rho_{n}}X_{s}\geq Y_{s}}$.
We fix $x\in\R_{+}$ and we want to find a formula for $p_{i}:=\mathbb{Q}_{\delta_{x},k}(\mathcal{A}_{i}\cap\mathcal{B})$.
By the Markov property we get 
\[
p_{i}=\mathbb{Q}_{\delta_{x},k}\sbr{\mathbb{Q}_{\delta_{x},k}(1_{\mathcal{A}_{i}}|X_{\rho_{k+1}})\mathbb{Q}_{\delta_{x},k}(1_{\mathcal{B}}|X_{\rho_{k+1}})}.
\]
We denote $L_{k}(x,y;i)\eqdef\mathbb{Q}_{\delta_{x},k}(1_{\mathcal{A}_{i}}|X_{\rho_{k+1}}=y)$,
which expresses more explicitly as 
\begin{equation}
L_{k}(x,y;i)=\pp\rbr{\forall_{\rho_{k}\leq s\leq\rho_{k+1}}X_{s}\geq\tilde{Y}_{s}^{k}(i)|Y^{k},X_{\rho_{k}}=x,X_{\rho_{k+1}}=y}.\label{eq:Ldef}
\end{equation}
We write
\[
p_{i}=\mathbb{Q}_{\delta_{x},k}\sbr{L(x,X_{\rho_{k+1}};i)\frac{\mathbb{Q}_{\delta_{x},k}(1_{\mathcal{B}}|X_{\rho_{k+1}})}{\mathbb{Q}_{\delta_{x},k}(1_{\mathcal{B}})}}\times\mathbb{Q}_{\delta_{x},k}(\mathcal{B}).
\]
Let us now consider a measure defined by 
\[
\mu_{x,n,k}(X_{\rho_{k+1}}\in D)\eqdef\mathbb{Q}_{\delta_{x},k}\sbr{1_{X_{\rho_{k+1}}\in D}\frac{\mathbb{Q}_{\delta_{x},k}(1_{\mathcal{B}}|X_{\rho_{k+1}})}{\mathbb{Q}_{\delta_{x},k}(1_{\mathcal{B}})}},
\]
where $D\subset\R$ is a Borel set. One easily verifies that it is
a probability measure. Removing the conditional expectation we get
\begin{multline*}
\mu_{x,n,k}(X_{\rho_{k+1}}\in D)=\mathbb{Q}_{\delta_{x},k}\sbr{\frac{\mathbb{Q}_{\delta_{x},k}(1_{X_{\rho_{k+1}}\in D}1_{\mathcal{B}}|X_{\rho_{k+1}})}{\mathbb{Q}_{\delta_{x},k}(1_{\mathcal{B}})}}=\frac{\mathbb{Q}_{\delta_{x},k}(1_{X_{\rho_{k+1}}\in D}1_{\mathcal{B}})}{\mathbb{Q}_{\delta_{x},k}(1_{\mathcal{B}})}\\
=\mathbb{Q}_{\delta_{x},k}\left(X_{\rho_{k+1}}\in D|\mathcal{B}\right).
\end{multline*}
Again, writing more explicitly we have
\begin{equation}
\mu_{x,n,k}(\dd y)=\pp\rbr{X_{\rho_{k+1}}\in\dd y|\forall_{\rho_{k+1}\leq s\leq\rho_{n}}X_{s}\geq Y_{s},X_{\rho_{k}}=x}.\label{eq:muxnk}
\end{equation}
Finally, concluding the above calculations we obtain that for $i\in\cbr{-1,1}$
we have 
\begin{align}
\pp_{m}\rbr{\forall_{\rho_{k}\leq s\leq\rho_{n}}X_{s}\geq\tilde{Y}_{s}^{k}(i)|Y^{k}} & =\int_{\R_{+}}\int_{\R_{+}}L_{k}(x,y;i)\mu_{x,n,k}(\dd{y)}m(\dd x)\label{eq:pDecomposition}\\
 & \quad\times\pp_{m}\rbr{\forall_{\rho_{k+1}\leq s\leq\rho_{n}}X_{s}\geq Y_{s}}.\nonumber 
\end{align}
Before going further let us comment on the further strategy. It is
easy to see that for any fixed $x,y\in\R_{+}$ we have $L_{k}(x,y;1)<L_{k}(x,y;-1)$.
The gap vanished however smaller when $x,y\to+\infty$. The uniform
inequality $g_{n,k}<1$ can be obtained by showing that with positive
probability the measure $\mu_{x,n,k}(\dd y)m(\dd x)$ is uniformly
concentrated in a box. 

Let $C_{1}>0$ be a constant as in Fact \ref{lem:boundedOU}. We denote
sequences of events
\begin{align*}
\mathcal{A}_{n,k} & \eqdef\mathcal{A}_{k}^{1}\cap\cbr{r_{k}\in[1,10]}\cap\mathcal{A}_{n,k}^{2},\\
\mathcal{A}_{k}^{1} & \eqdef\cbr{\rho_{k}\geq1}\cap\cbr{\forall_{s\leq\rho_{k}}Y_{s}\leq C_{1}b_{\rho_{k}}(s)},\\
\mathcal{A}_{n,k}^{2} & \eqdef\cbr{\pr{\mathcal{B}_{n,k}|\mathcal{F}}\geq1/10},
\end{align*}
where $\mathcal{B}_{n,k}\eqdef\cbr{\forall_{\rho_{k+1}\leq s\leq\rho_{n}}Y_{s}\leq C_{1}(s-\rho_{k+1}+1)^{1/3}}$.
We first prove concentration of $m=T_{k}/\norm{T_{k}}{}{}$ (recall
(\ref{eq:tmp99})). Let $R>0$, by the FKG property stated in Lemma
\ref{lem:StrongFKG}, conditionally on the event $\mathcal{A}_{k}^{1}$
we have
\begin{align*}
\rbr{T_{k}/\norm{T_{k}}{}{}}([R^{-1},R]) & =\pr{X_{\rho_{k}}\in[0,R]|\forall_{s\leq\rho_{k}}X_{s}\geq Y_{s},Y}-\pr{X_{\rho_{k}}\leq R^{-1}|\forall_{s\leq\rho_{k}}X_{s}\geq Y_{s},Y}\\
 & \geq\pr{X_{\rho_{k}}\in[R^{-1},R]|\forall_{s\leq\rho_{k}}X_{s}\geq C_{1}b_{\rho_{k}}(s)}-\pr{X_{\rho_{k}}\leq R^{-1}|X_{\rho_{k}}\geq0}.
\end{align*}
Using Lemma \ref{lem:uniformTightness} we can choose $R>0$ such
that the first term is arbitrarily close to $1$. By easy calculations
the second term can be made arbitrarily close to $0$.\textcolor{blue}{{}
}We fix $R$ such that 
\begin{equation}
\rbr{T_{k}/\norm{T_{k}}{}{}}([R^{-1},R])\geq1_{\mathcal{A}_{k}^{1}}(1/2).\label{eq:conditionalMeasure}
\end{equation}
Our next aim is to study concentration of (\ref{eq:muxnk}). To this
end we denote $$p(x)\eqdef\pp\rbr{X_{\rho_{k+1}}\in[0,C_{e}]|\forall_{\rho_{k+1}\leq s\leq\rho_{n}}X_{s}\geq Y_{s},X_{\rho_{k}}=x},$$
where $C_{e}>0$ is to be fixed later. Using Fact \ref{fact:OUfact}
we estimate 
\begin{align*}
p(x)&\geq  \pp\rbr{\cbr{X_{\rho_{k+1}}\in[0,C_{e}]}\cap\mathcal{B}_{n,k}|\forall_{\rho_{k+1}\leq s\leq\rho_{n}}X_{s}\geq Y_{s},X_{\rho_{k}}=x}\\
 &=  \pp\rbr{X_{\rho_{k+1}}\in[0,C_{e}]|\mathcal{B}_{n,k}\cap\cbr{\forall_{\rho_{k+1}\leq s\leq\rho_{n}}X_{s}\geq Y_{s},X_{\rho_{k}}=x}}\\
 &\qquad \qquad \qquad\qquad \qquad \qquad\qquad \qquad \qquad \times \pp\rbr{\mathcal{B}_{n,k}|\forall_{\rho_{k+1}\leq s\leq\rho_{n}}X_{s}\geq Y_{s},X_{\rho_{k}}=x}\nonumber \\
&\geq  \pp\rbr{X_{\rho_{k+1}}\in[0,C_{e}]|\forall_{\rho_{k+1}\leq s\leq\rho_{n}}X_{s}\geq C_{1}(s-\rho_{k+1}+1)^{1/3},X_{\rho_{k}}=x}\pp\rbr{\mathcal{B}_{n,k}}.\nonumber 
\end{align*}
Assuming that $r_{k}=\rho_{k+1}-\rho_{k}\in[1,10]$, we can set
$C_{e}$ such that Lemma \ref{lem:forwardConditioningConcerntration}
implies that the first term is bounded away from $0$ uniformly in
$n,k$ and $x\in[R^{-1},R]$. Together with (\ref{eq:conditionalMeasure})
this implies 
\[
\int_{R^{-1}}^{R}\mu_{x,n,k}([0,C_{e}])m(\dd x)\geq C1_{\mathcal{A}_{n,k}},
\]
for some $C>0$. One further finds $c_{e}>0$ (we skip details) such
that

\begin{equation}
\int_{R^{-1}}^{R}\mu_{x,n,k}([c_{e},C_{e}])m(\dd x)\geq(C/2)1_{\mathcal{A}_{n,k}}.\label{eq:abc}
\end{equation}
We are now ready to come back to (\ref{eq:gnk}). We recall (\ref{eq:Ldef})
and denote

\begin{equation}
r_{k}\eqdef\sup_{x\in[R^{-1},R],y\in[c_{e},C_{e}]}\frac{L_{k}(x,y;1)}{L_{k}(x,y;-1)},\quad L_{k}\eqdef\inf_{x\in[R^{-1},R],y\in[c_{e},C_{e}]}L_{k}(x,y;-1).\label{eq:defr}
\end{equation}
We have $L_{k}(x,y;-1)>L_{k}(x,y;1)$. Moreover, one verifies that
for $\eta\in\{-1,1\}$ the functions $L_{k}(\cdot,\cdot;\eta)$ are
continuous and $L_{k}(x,y;\eta)>0$ if $x,y>0$. These imply $r_{k}<1$
(note that $r_{k}$ is a random variable since it depends on $Y^{k}$).
Similarly we have $L_{k}>0$. One checks that 
\[\int_{\R_{+}}\int_{\R_{+}}L_{k}(x,y;i)\mu_{x,k,n}(\dd{y)}m(\dd x)\leq1.\]
Using the elementary inequality $(a+c)/(b+d)\leq(a+1)/(b+1)$ valid
for $0<a\leq b$ and $0<c\leq d\leq1$ we get
\[
g_{n,k}\leq\frac{\int_{R^{-1}}^{R}\int_{c_{e}}^{C_{e}}L_{k}(x,y;1)\mu_{x,n,k}(\dd{y)}m(\dd x)+1}{\int_{R^{-1}}^{R}\int_{c_{e}}^{C_{e}}L_{k}(x,y;-1)\mu_{x,n,k}(\dd{y)}m(\dd x)+1}\leq\frac{r_{k}\int_{R^{-1}}^{R}\int_{c_{e}}^{C_{e}}L_{k}(x,y;-1)\mu_{x,n,k}(\dd{y)}m(\dd x)+1}{\int_{R^{-1}}^{R}\int_{c_{e}}^{C_{e}}L_{k}(x,y;-1)\mu_{x,n,k}(\dd{y)}m(\dd x)+1}.
\]
Further, we notice that for $r_{k}\in[0,1]$ we have $(r_{k}a+1)/(a+1)\leq(r_{k}b+1)/(b+1)$
if $b\leq a$. Applying (\ref{eq:abc}) we get
\[
g_{n,k}\leq\frac{r_{k}L_{k}\int_{R^{-1}}^{R}\mu_{x,n,k}([c_{e},C_{e}])m(\dd x)+1}{L_{k}\int_{R^{-1}}^{R}\mu_{x,n,k}([c_{e},C_{e}])m(\dd x)+1}\leq\frac{r_{k}L_{k}C/2+1}{L_{k}C/2+1}\leq1+\frac{CL_{k}\rbr{r_{k}-1}}{4},
\]
where the last estimate follows by $(ab+1)/(a+1)\leq1+a(b-1)/2$ valid
for $a,b\in[0,1]$. Combining the last inequality with (\ref{eq:DeltaGapEstimate})
we arrive at 
\begin{equation}
G_{n,k}\leq-C_{3}\E_{k}\sbr{1_{\mathcal{A}_{n,k}}L_{k}^{2}\rbr{r_{k}-1}^{2}}.\label{eq:gtildeDef}
\end{equation}
for a constant $C_{3}>0$. \textcolor{blue}{}\global\long\def\ev#1{\mathbb{E}{#1}}
\textcolor{blue}{}\global\long\def\E{\mathbb{E}}
\textcolor{blue}{}\global\long\def\pp{\mathbb{P}}
\textcolor{blue}{}\global\long\def\pr#1{\mathbb{P}\rbr{#1}}

We are now ready to show (\ref{eq:gapGoal}) (which concludes the
proof of the proposition). By the strong Markov property we have
\[
\E G_{n,k}\leq-C_{3}\E\sbr{1_{\mathcal{A}_{n,k}}L_{k}^{2}\rbr{r_{k}-1}^{2}}\leq-C_{3}\pr{\mathcal{A}_{k}^{1}}\E\sbr{L_{k}^{2}\rbr{r_{k}-1}^{2}1_{r_{k}\in[1,10]}}\pr{\mathcal{A}_{n,k}^{2}}.
\]
 We notice that the law of $L_{k}^{2}\rbr{r_{k}-1}^{2}1_{r_{k}\in[1,10]}$
does not depend on $k$ and it is not concentrated on $0$, thus $\pr{\mathcal{A}_{k}^{1}}\E\sbr{L_{k}^{2}\rbr{r_{k}-1}^{2}1_{r_{k}\in[1,10]}}>0$.
The other two terms are uniformly bounded away from $0$ by Fact \ref{lem:boundedOU}.
\end{proof}

\begin{proof}[Proof of Theorem~\ref{thm:OUresult}]
As observed at the beginning of the section, Lemmas~\ref{lem:proofGammaDef} and \ref{lem:fungamma}, as well as Lemmas~\ref{lem:proofDeltaDef} and \ref{lem:fundelta} respectively prove there exist two convex symmetric functions $\gamma,\delta$ such that for all $\beta \in \R$,
\[  \lim_{t\to+\infty}\frac{\log\mathbb{P}\rbr{\forall_{s\leq t}X_{s}\geq\beta Y_{s}, X_t - \beta Y_t \in (a,b)|Y}}{t} = -\gamma(\beta)\]
\[\lim_{t\to+\infty}\frac{\log\mathbb{P}\rbr{\forall_{s\leq t}X_{s}\geq\beta Y_{s},X_t - \beta Y_t \in (a,b)}}{t}= -\delta(\beta).\]

Moreover, by Lemma~\ref{lem:relation} and Proposition~\ref{prop:mainPropositon}, for any $\beta \neq 0$, there exists $c>0$ such that
\[
  \gamma(\beta) \geq \delta(\beta) + c,
\]
proving that $\gamma(\beta)>\delta(\beta)$ for any $\beta \neq 0$.
\end{proof}

\section{Proof of Theorems \ref{thm:BMBasic} and \ref{thm:BMExtended}}
\label{sec:BMProofs}

As observed in Remark \ref{rem:BMtoOU}, Theorem \ref{thm:BMBasic} is a direct consequence of Theorem \ref{thm:OUresult}. We now extend this result to prove Theorem \ref{thm:BMExtended}.

\begin{proof}[Proof of Theorem \ref{thm:BMExtended}]
For simplicity, we assume here that $a=0$ and $b=+\infty$, therefore we ignore the condition $B_t - \beta W_t \in (a t^{1/2}, b t^{1/2})$. The proof in the general case is obtained in the same fashion, but with heavier notation.

Denoting $x\eqdef\inf_{t\geq0}g(t)$,
one can find $A>0$ and $\epsilon>0$ such that 
\[
-x/2-Aj(t)\leq f(t)\leq x/2+Aj(t),\quad t\geq0,
\]
where $j(t)\eqdef\min(t,t^{1/2-\epsilon})$. We have 
\begin{multline}
\pr{\forall_{s\leq t}x/2+B_{s}\geq W_{s}+Aj(s)|W}\leq\pr{\forall_{s\leq t}x+B_{s}\geq W_{s}+f(s)|W}\\
\leq\pr{\forall_{s\leq t}g(t)+B_{s}\geq W_{s}+f(s)|W}\leq\pr{\forall_{s\leq t}g(t)+B_{s}\geq W_{s}-Aj(s)|W}\\
\leq\pr{\forall_{h(t)\leq s\leq t}g(t)+x/2+B_{s}\geq W_{s}-Aj(s)|W},\label{eq:sandwich}
\end{multline}
where $h(t)$ is any function such that $(h(t))^{1/2-\epsilon}\geq g(t)+x+1$
and $h(t)=e^{o(\log t)}$. The right-hand side of the last expression
is bounded from above by $\pr{\forall_{h(t)\leq s\leq t}1+B_{s}\geq W_{s}-(A+1)j(s)|W}$.
We will show that the event

\begin{equation}
\mathcal{A}\eqdef\cbr{\lim_{t\to+\infty}-\frac{\log\pr{\forall_{h(t)\leq s\leq t}1+B_{s}\geq W_{s}-Aj(s)|W}}{\log t}=\gamma(\beta)},\label{eq:assertion1}
\end{equation}
fulfills $\pr{\mathcal{A}}=1$. The same method can be used to show
the almost sure convergence of the left-hand side of (\ref{eq:sandwich})
(we skip the details). These will conclude the proof. We define a stochastic process $\cbr{Z_{t}}_{t\geq0}$ by 
\[
Z_{t}\eqdef\exp\rbr{\int_{0}^{t}Aj'(s)\dd W_{s}-\frac{1}{2}\int_{0}^{t}[Aj'(s)]^{2}\dd s}.
\]
This process is an uniformly integrable martingale (since $\int_{0}^{+\infty}[j'(s)]^{2}\dd s\leq+\infty$).
We denote its limit by $Z_{\infty}$, that is $\mathbb{P}$-a.s. positive, and define the measure (on the
Wiener space) $\dd{\mathbb{Q}}=Z_{\infty}\dd{\mathbb{P}}.$ By the
Girsanov theorem under this measure $\cbr{W_{s}-Aj(s)}_{s\geq0}$
is a standard Wiener process. We prove that
\begin{equation}
\mathbb{Q}(\mathcal{A})=1,\label{eq:toBeProven}
\end{equation}
which is enough to conclude the proof. Indeed, as $\pr{Z_{\infty}>0}=1$ and $\ev{}1_{\mathcal{A}}Z_{\infty}=\mathbb{Q}(\mathcal{A})=1$, we have
 $\pr{\mathcal{A}}=1$. We are now going to show (\ref{eq:toBeProven}).
By Theorem \ref{thm:BMBasic} it is enough to prove 
\[
0\leq\delta_{t}\eqdef-\rbr{\frac{\log\pr{\forall_{0\leq s\leq t}1+B_{s}\geq W_{s}|W}}{\log t}-\frac{\log\pr{\forall_{h(t)\leq s\leq t}1+B_{s}\geq W_{s}|W}}{\log t}}\to0,\quad\mathbb{P}-\text{a.s.}
\]
We have
\begin{align*}
\delta_{t} & =-\frac{\log\mathbb{P}\rbr{\forall_{s\leq h(t)}1+B_{s}\geq W_{s}|\forall_{h(t)\leq s\leq t}1+B_{s}\geq W_{s},W}}{\log t}\\
 & \leq-\frac{\log\mathbb{P}\rbr{\forall_{s\leq h(t)}1+B_{s}\geq W_{s}|W}}{\log t},
\end{align*}
where the last inequality holds by Lemma \ref{lem:StrongFKG}.
Now the proof is straightforward, indeed 
\[
-\frac{\log\rbr{\forall_{0\leq s\leq h(t)}1+B_{s}\geq W_{s},W}}{\log t}=-\frac{\log\mathbb{P}\rbr{\forall_{0\leq s\leq h(t)}1+B_{s}\geq W_{s},W}}{\log h(t)}\frac{\log h(t)}{\log t}\to0,\quad\mathbb{P}-\text{a.s.}
\]
Theorem \ref{thm:BMBasic} and $\log h(t)/\log t\to0$. 

In order to prove the $L^{p}$ convergence it is enough to show that the family 
\[
\cbr{-\log\pr{\forall_{s\leq t}x/2+B_{s}\geq W_{s}+Aj(s)|W}/t}_{t\geq0},
\]
 is $L^{p}$-uniformly integrable. This follows by easy calculations
from Fact \ref{fact:uniformIntegrability} using relation (\ref{eq:OUandBMRelation}). 
\end{proof}

We end this section with a proof of Fact \ref{fact:IIDwall}.
\begin{proof}[Proof of Fact \ref{fact:IIDwall}]
Let $(X_k)_{k \geq 0}$ be a sequence of i.i.d. random variables such that $\E(X_k^{2+\epsilon})<+\infty$. By Borel-Cantelli lemma, we have
\[
  L \eqdef \max\left\{n \in \N : |X_n| \geq n^{1/(2+\epsilon)} \right\} < +\infty \text{ a.s.}
\]
Therefore, we have
\[
  \limsup_{N \to +\infty} \frac{\log \pr{\forall_{n \leq N} B_n \geq X_n | X}}{\log n}
  \leq \limsup_{N\to +\infty} \frac{\log \pr{\forall_{L < n \leq N} B_n \geq n^{1/(2+\epsilon)}|L}}{\log n} \leq \frac{-1}{2} \text{a.s.}
\]
Similarly, setting $M=\max_{n \leq L} X_n$, we have
\[  \liminf_{N \to +\infty} \frac{\log \pr{\forall_{n \leq N} B_n \geq X_n | X}}{\log n}
  \geq \liminf_{N\to +\infty} \frac{\log \pr{\forall_{n \leq N} B_n \geq -M-n^{1/(2+\epsilon)}|L,M}}{\log n} \geq \frac{-1}{2} \text{a.s.}
\]
\end{proof}

\section{Proof of the facts for random walks}
\label{sub:ProofsRW}

We now use definitions of Section \ref{sub:Dependent-version-of}. We write $S$ a random walk in random environment $\mu$, and set
\[
  W_n = - \mathbb{E}( S_n | \mu) \quad \text{ and } \quad B_n = S_n + W_n.
\]
To make the notation more clear in this section we assume that we have two probability spaces $(\Omega,\mathcal{F},\mathbb{P})$, $(\hat{\Omega},\mathcal{G},\mathbb{Q})$
which supports $B$ and $\mu$ respectively. The~measure $\mathbb{P}$
depends on the realization of $\mu$, which is made implicit in the notation.
Thus we are going to prove 
\begin{multline}
	\lim_{N\to+\infty}\frac{\log\pr{\forall_{n\leq N}x+B_{n}\geq W_{n}+f(n), B_N - W_N\in (aN^{1/2}, b N^{1/2}) }}{\log N}\\=\begin{cases}
	-\gamma\rbr{\sqrt{\frac{\Var{W_{1}}}{\Var{B_{1}}}}} & \text{ on }\mathcal{A}_{x},\\
	-\infty & \text{ on }\mathcal{A}_{x}^{c}.
	\end{cases},\quad\mathbb{Q}\text{-a.s.}\label{eq:aimAim}	
\end{multline}
Further to simplify the notation we put $\gamma\eqdef\gamma\rbr{\sqrt{\frac{\Var{W_{1}}}{\Var{B_{1}}}}}$
and 
\begin{equation}
p_{N}\eqdef\pr{\forall_{n\leq N}x+B_{n}\geq W_{n}+f(n), B_N - W_N \in (aN^{1/2}, bN^{1/2})}.\label{eq:pN}
\end{equation}
We need a bound that the inhomogeneous random walk $B$ grows fast.
This will be contained in the first two lemmas of this section. We
will use tilting of measure. Let us denote the increments of $B$
by 
\[
X_{n}\eqdef B_{n}-B_{n-1}.
\]
Let us recall $C_{1}$ from the assumption (A2). For any $\theta\in[0,C_{1}]$
and $n\in\N$ we define a probability measure $\mathbb{H}_{n,\theta}$
by 
\begin{equation}
\frac{\dd{\mathbb{H}}_{n,\theta}}{\dd{\pp}}\eqdef e^{\theta X_{n}}\psi_{n}^{-1}(\theta),\quad\psi_{n}(\theta)=\ev{}e^{\theta X_{n}}.\label{eq:tiltedMeasure}
\end{equation}
The tilting is supposed to ``increase'' $X_{n}'$s. The following
lemma quantifies this 
\begin{lem}
\textup{\label{lem:beastofRevelation}There exist $\theta_{0}\in(0,C_{1})$
and $0<c\leq\tilde{c}$ such that for any $\theta\leq\theta_{0}$
and $n\in\mathbb{N}$ we have
\[
\mathbb{H}_{n,\theta}(X_{n})\in[c\theta\E X_{n}^{2},\tilde{c}\theta\E X_{n}^{2}].
\]
}\end{lem}
\begin{proof}
By (\ref{eq:tiltedMeasure}) we have
\[
\mathbb{H}_{n,\theta}(X_{n})=\psi_{n}^{-1}(\theta)\ev{}X_{n}e^{\theta X_{n}}.
\]
The proof will be finished once we show that for any $n$ and small
enough $\theta$ we have 
\begin{equation}
\ev{}X_{n}e^{\theta X_{n}}=(\ev{}X_{n}^{2})\theta+o(\theta),\quad\psi_{n}(\theta)=1+O(\theta^{2}).\label{eq:questions}
\end{equation}
By the assumption of the uniform exponential integrability in (A2)
and Cauchy's estimate \cite[Theorem 10.26]{Rudin:1987lr} for any
$n$ and $0\leq\theta\leq C_{1}/2$ we get 
\[
\psi_{n}''(\theta)\leq C,\quad|\psi_{n}'''(\theta)|\leq C,
\]
for some $C>0$. By the assumptions $\psi_{n}(0)=1$ and $\psi_{n}'(0)=\E X_{n}=0$,
thus the second statement of (\ref{eq:questions}) follows by the
Taylor formula (with the Lagrange reminder). For the first one we
notice that $\ev{}X_{n}e^{\theta X_{n}}=\psi_{n}'(\theta)$, $\psi_{n}''(0)=\ev{}X_{n}^{2}$
and again apply the Taylor expansion. 
\end{proof}
We present now the aforementioned bound. 
\begin{lem}
\label{thm:veryFast}There exist $c,C>0$ such that for large
enough $N$ on the event 
\[
\cbr{\sum_{n=1}^{N}\ev{}(X_{n}^{2})\geq N\mathbb{Q}(\ev{}X_{i}^{2})/2}
\]
we have
\[
\pr{B_{N}\geq c\sqrt{N}\log\log N}\geq N^{-C(\log\log N)^{2}}.
\]
\end{lem}
\begin{proof}
We define $a_{n}\eqdef\frac{1}{4}\theta_{0}cn^{1/2}\log\log n$, where
$\theta_{0},c$ are given by Lemma \ref{lem:beastofRevelation} and
consider the events $\mathcal{A}_{N}\eqdef\cbr{S_{N}\geq a_{N}}.$ We
denote also $b_{n}\eqdef(\theta_{0}n^{-1/2}\log\log n)\vee0$ and
let us define the tilted measure $\mathbb{P}_{N}$ by
\begin{equation}
\frac{\dd{\mathbb{P}_{N}}}{\dd{\mathbb{P}}}\eqdef\Lambda_{N},\quad\Lambda_{N}\eqdef\exp\rbr{\sum_{n=1}^{N}b_{n}X_{n}}\prod_{n=1}^{N}\psi(b_{n})^{-1}.\label{eq:tiltedMeasureDef}
\end{equation}
Further, we write $q_{N}\eqdef\pr{\mathcal{A}_{N}}=\mathbb{P}_{N}\rbr{1_{\mathcal{A}_{N}}\Lambda_{N}^{-1}}.$
We have to estimate 
\[
q_{N}=\mathbb{P}_{N}\rbr{1_{\mathcal{A}_{N}}\Lambda_{N}^{-1}}=\rbr{\prod_{n=1}^{N}\psi(b_{n})}\mathbb{P}_{N}\rbr{1_{\mathcal{A}_{N}} \exp\rbr{-\sum_{n=1}^{N}b_{n}X_{n}}}\geq\mathbb{P}_{N}\rbr{1_{\mathcal{A}_{N}}\exp\rbr{-\sum_{n=1}^{N}b_{n}X_{n}}}.
\]
We introduce $\tilde{X}_{n}\eqdef X_{n}-\E_{N}X_{n}$ and accordingly
$\tilde{B}_{n}\eqdef\sum_{i=1}^{n}\tilde{X}_{i}$. In our notation
\[
q_{N}\geq\exp\rbr{-\sum_{i=1}^{N}(\mathbb{E}_{N}X_{n})b_{n}}\mathbb{P}_{N}\rbr{1_{\mathcal{A}_{N}}\exp\rbr{-\sum_{n=1}^{N}b_{n}\tilde{X}_{n}}}.
\]
Now, by Lemma \ref{lem:beastofRevelation} and the assumption (A2) we obtain 
\[
\sum_{n=1}^{N}(\E_{N}X_{n})b_{n}\leq C_1 \sum_{n=1}^{N}b_{n}^{2}\leq 2C_{1}(\log N)(\log\log N)^{2},
\]
for $C_{1}>0$. Next, we apply the Abel transform 
\[
\sum_{n=1}^{N}b_{n}\tilde{X}_{n}=\sum_{n=1}^{N}b_{n}(\tilde{B}_{n}-\tilde{B}_{n-1})=\tilde{B}_{N}b_{N}+\sum_{n=1}^{N-1}(b_{n}-b_{n+1})\tilde{B}_{n}.
\]
We define events $\mathcal{B}_{N}\eqdef\cbr{\forall_{n\leq N}|\tilde{B}_{n}|\leq C_{2}a_{n}}$,
for some $C_{2}>0$. We have an elementary estimation $|b_{n}-b_{n+1}|\leq C_{3}n^{-3/2}\log\log n$,
$C_{3}>0$. Putting things together we obtain 
\begin{align*}
q_{N} & \geq N^{-C_{1}(\log\log N)^{2}}\mathbb{P}_{N}\rbr{1_{\mathcal{A}_{N}\cap\mathcal{B}_{N}}\exp\rbr{-|\tilde{B}_{N}b_{N}|-C_{3}\sum_{n=1}^{N-1}|n^{-3/2}\log\log n)||\tilde{B}_{n}|}}\\
 & \geq N^{-C_{1}(\log\log N)^{2}}\mathbb{P}_{N}\rbr{1_{\mathcal{A}_{N}\cap\mathcal{B}_{N}}\exp\rbr{-\tilde{B}_{N}b_{N}-C_{4}\sum_{n=1}^{N-1}n^{-1}(\log\log n)^{2}}}\\
 & \ge N^{-C_{5}(\log\log N)^{2}}\mathbb{P}_{N}\rbr{\mathcal{A}_{N}\cap\mathcal{B}_{N}},
\end{align*}
where we introduced $C_{4},C_{5}>0$. We notice that 
\[
\mathcal{A}_{N}\supseteq\cbr{\tilde{B}_{N}\geq a_{N}-\sum_{n=1}^{N}\E_{N}X_{i}}\supseteq\cbr{\tilde{B}_{N}\geq-a_{N}/2}.
\]
Finally, we leave to the reader proving that $\liminf_{N\to\infty}\mathbb{P}_{N}(\mathcal{A}_{N}\cap\mathcal{B}_{N})>0,$
which concludes the proof. 
\end{proof}
Let us recall the event $\mathcal{A}_{x}$ defined in (\ref{eq:canStay2}). Let us denote by $\tilde{p}_N$ the version of 
$p_{N}$ from (\ref{eq:pN}) without condition $B_N - W_N \in (aN^{1/2}, bN^{1/2})$ i.e.
\begin{equation}
\tilde p_{N}\eqdef\pr{\forall_{n\leq N}x+B_{n}\geq W_{n}+f(n)}.\label{eq:tpN}
\end{equation}
In the following lemma we prove a crude bound corresponding to the bound from above in Theorem~\ref{thm:BRWBasicExtended}. Namely
\begin{lem}
\label{lem:crudeLowerBound}We have 
\[
\liminf_{N\to+\infty}\frac{\log \tilde p_{N}}{N^{2}}\geq\begin{cases}
0 & \text{ on }\mathcal{A}_{x},\\
-\infty & \text{ on }\mathcal{A}_{x}^{c}.
\end{cases}\quad\mathbb{Q}-a.s.
\]
\end{lem}
\begin{proof}
The proof will follow again by the change of measure techniques. Due
to a very big normalization the proof can be somewhat brutal. We fix
$b_{N}=b\in(0,C_{1})$ ($C_{1}$ as in (A2)) and use $\Lambda_{N}$
and $\mathbb{P}_{N}$ as in (\ref{eq:tiltedMeasureDef}). We denote
$\mathcal{B}_{N}\eqdef\cbr{\forall_{n\leq N}x+B_{n}\geq W_{n}+f(n)}$ and
calculate 
\[
 \tilde p_{N}= \mathbb{P}_{N}(1_{\mathcal{B}_{N}}\Lambda_{N}^{-1})\geq\mathbb{P}_{N}\rbr{1_{\mathcal{B}_{N}}\exp\rbr{-\sum_{n=1}^{N}bX_{n}}}.
\]
We introduce also $\mathcal{B}_{N}^{1}\eqdef\cbr{\forall_{n\in\cbr{1,\ldots,N}}X_{n}\leq N^{1/2}}$. Trivially we have
\begin{align}
\tilde p_{N} & \geq\mathbb{P}_{N}\rbr{1_{\mathcal{B}_{N}\cap\mathcal{B}_{N}^{1}}\exp\rbr{-\sum_{n=1}^{N}bX_{n}}}\geq e^{-bN^{3/2}}\mathbb{P}_{N}(\mathcal{B}_{N}\cap\mathcal{B}_{N}^{1})\label{eq:uff}\\
 & \geq e^{-bN^{3/2}}\sbr{\mathbb{P}_{N}(\mathcal{B}_{N})-\mathbb{P}_{N}((\mathcal{B}_{N}^{1})^{c})}\geq e^{-bN^{3/2}}\sbr{\mathbb{P}_{N}(\mathcal{B}_{N})-Ne^{-cN^{1/2}}},\nonumber 
\end{align}
where the last inequality follows by the union bound and the fact
that exponential moments of $X_{n}$ are uniformly bounded, see (A2). Let us
concentrate on $\mathbb{P}_{N}(\mathcal{B}_{N})$. We denote $v=\mathbb{Q}\sbr{\ev{}(X_{i}^{2})}$
and define 
\[
L\eqdef\inf\cbr{n\geq0:\forall_{k\geq n}\sum_{i=0}^{k}\ev{}(X_{i}^{2})\geq kv/2}.
\]
Clearly $\mathbb{Q}(L<+\infty)=1$. Fix $K>0$ and denote the following
events in $\mathcal{G}$ (i.e. describing conditions on $W$) 
\[
\mathcal{A}_{K}\eqdef\mathcal{A}_{K}^{1}\cap\mathcal{A}_{K}^{2},\quad\mathcal{A}_{K}^{1}\eqdef\cbr{L<K},\quad\mathcal{A}_{K}^{2}\eqdef\cbr{\forall_{k\geq K}W_{k}\leq k^{2/3} - f(k)}.
\]
Using the Markov property we get
\begin{align}
\mathbb{P}_{N}(\mathcal{B}_{N}) & \geq\mathbb{P}_{N}\rbr{{\forall_{n\in\cbr{1,\ldots,K}}x+B_{n}\geq W_{n}+f(n)},{B_{K}\geq Kv/2}}\label{eq:start}\\
 & \quad\times\mathbb{P}_{N}\rbr{\forall_{n\in\cbr{K,\ldots,N}}B_{n}\geq W_{n}+f(n)|B_{K}=Kv/2}.\nonumber 
\end{align}
We denote the first term by $\tilde p_{K}$. It is easy to check that the
law of $B_{n}$ under $\mathbb{P}_{N}$ stochastically dominates the
one under $\mathbb{P}$ thus $\cbr{p_{K}=0}\subset\mathcal{A}_{x}^{c}.$
Conditionally on $\mathcal{A}_{K}^{2}$ we have 
\[
\mathbb{P}_{N}\rbr{\forall_{n\in\cbr{K,\ldots,N}}B_{n}\geq W_{n}+f(n)|B_{K}=Kv/2}\geq1-\sum_{n=K}^{+\infty}\mathbb{P}_{N}\rbr{B_{n}\leq nv/4|B_{K}=Kv/2}.
\]
We denote $\tilde{B}_{n}\eqdef B_{n}-\mathbb{E}_{N}B_{n}$. By Lemma
\ref{lem:beastofRevelation}, conditionally on $\mathcal{A}_{K}^{1}$,
we have $\mathbb{E}_{N}B_{k}\geq c kv/2$, thus 
\[
\mathbb{P}_{N}\rbr{\forall_{n\in\cbr{K,\ldots,N}}B_{n}\geq W_{n}+f(n)|B_{K}=Kv/2}\geq1-\sum_{n=0}^{+\infty}\mathbb{P}_{N}\rbr{\tilde{B}_{n}\leq-cnv/4|B_{0}=Kv/2}.
\]
Observing that the random variables $X_{n}$ are uniformly exponentially integrable we get a constant $c_1>0$ such that 
\[
\mathbb{P}_{N}\rbr{\forall_{n\in\cbr{K,\ldots,N}}B_{n}\geq W_{n}+f(n)|B_{K}=Kv/2}\geq1-\sum_{n=K}^{+\infty}e^{-c_1n}>0,
\]
for $K$ large enough. Putting the above estimates to (\ref{eq:start})
we obtain that for some $C>0$ 
\[
1_{\mathcal{A}_{K}}\mathbb{P}_{N}(\mathcal{B}_{N})\geq1_{\mathcal{A}_{K}}C\tilde p_{K}.
\]
Using this in (\ref{eq:uff}) we have
\[
\liminf_{N\to+\infty}\frac{\log \tilde p_{N}}{N^{2}}\geq\begin{cases}
0 & \text{ on }\mathcal{A}_{K}\cap\mathcal{A}_{x},\\
-\infty & \text{ on }\mathcal{A}_{K}^{c}\cup\mathcal{A}_{x}^{c}.
\end{cases}\quad\mathbb{Q}-a.s.
\]
The proof is concluded passing $K\nearrow+\infty$ and by observing
that $1_{\mathcal{A}_{K}}\to1$, $\mathbb{Q}$-a.s.
\end{proof}
We finally pass to the proof of Theorem \ref{thm:BRWBasicExtended}. We notice that by the very definition of $\mathcal{A}_{x}$ (see \ref{eq:canStay2})
it is obvious that the convergence holds on $\mathcal{A}_{x}^{c}$.
Thus in the proofs below we concentrate on proving the convergence
on the event $\mathcal{A}_{x}$. The instrumental tool of this proof
will be the so-called KMT coupling. We choose the measure $\mathbb{Q}$
to be a special one. By \cite[Corollary 2.3]{Lifshits:2000fj} we
can find a probability space $(\tilde{\Omega},\mathcal{G},\mathbb{Q})$
with processes $\cbr{W_{k}}_{k\geq0}$ and $\cbr{\hat{W}_{k}}_{k\geq0}$, which is a random walk with the increments distributed according to $\mathcal{N}(0,\sigma_{W}^{2})$,
where $\sigma_{W}^{2}=\Var{W_{1}}$ and
\begin{equation}
\limsup_{k\to+\infty}\frac{|\hat{W}_{k}-W_{k}|}{(\log k)^{2}}=0,\quad\mathbb{Q}-\text{a.s.}\label{eq:approximation}
\end{equation}
Further we can extend the measure $\mathbb{Q}$ so that  $\hat W$ is a marginal of a Brownian motion, also denoted by $\hat W$. Slightly abusing notation we keep $\mathbb{Q}$ to denote the extended
measure and $\hat{W}$ for the Brownian motion. 

First we prove (\ref{eq:aimAim}) for this special measure. At the
end of the proof we argue how this statement implies the thesis of Theorem \ref{thm:BRWBasicExtended}. We start with the bound from above. We recall $p_N$ defined in \eqref{eq:pN}. One finds $A$ and $\epsilon>0$ such that for any $n\in \N$ we have $f(n)\geq - A n^{1/2-\epsilon}$. Further we set $\tilde p_N$ of \eqref{eq:tpN} with this function, i.e. 
\[
	 \tilde p_{N}\eqdef\log\pr{\forall_{n\leq N}x+B_{n}\geq W_{n}-A n^{1/2-\epsilon}}.
\]
In this part we will show 
\begin{equation}
\limsup_{N\to+\infty}\frac{\log p_{N}}{\log N}\leq \limsup_{N\to+\infty}\frac{\log \tilde p_{N}}{\log N} \leq-\gamma,\quad\mathbb{Q}\text{-a.s.}\label{eq:smallAim-1}
\end{equation}
We define a function $f:\mathbb{N} \times \mathcal{M}(\R)^\mathbb{N} \mapsto\R$ by 
\begin{equation}
f(n,\mu)\eqdef\sum_{i=1}^{n}\ev{}X_{i}^{2}.\label{eq:timeChange}
\end{equation}
We recall that the measure $\mathbb{P}$ depends on realization of
$W$ and that $\cbr{Z_{i}\eqdef\ev{}X_{i}^{2}}_{i\geq0}$ is a sequence
of i.i.d variables with respect to $\qq$. It is easy to check, using
the exponential Chebyshev inequality, that (A2) implies existence
of $\tilde{C}_{1}>0$ such that $Z_{i}\leq\tilde{C}_{1}$. We define
a sequence of events belonging to $\mathcal{G}$ given by
\begin{align}
\mathcal{A}_{N}\eqdef & \cbr{\forall_{k\leq N}|\hat{W}_{k}-W_{k}|\leq(\log N)^{2}}\cap\cbr{\forall_{k\leq N}\sup_{t\in[k,k+1]}|\hat{W}_{t}-\hat{W}_{k}|\leq(\log N)^{2}}\label{eq:eventAN}\\
 & \cap\cbr{\forall_{k\geq\log N}\max_{l\in\cbr{-k^{2/3},\ldots,k^{2/3}}}|\hat{W}_{k}-\hat{W}_{k+l}|\leq k^{4/9}}\nonumber \\
 & \cap\cbr{\forall_{k\geq\log N}|f(k,\mu)-k\qq(\ev{}(X_1^{2}))|\leq k^{1/2}\log k}.\nonumber 
\end{align}
We have $1_{\mathcal{A}_{N}}\to1$ $\qq$-a.s. The convergence of
the first term follows by (\ref{eq:approximation}). The proof of
the others are rather standard (we note that exponents $2/3$ and
$4/9$ can be made smaller but this is not relevant for our proof).
As an example we treat the last but one term. We set 
\[
q_k\eqdef\pr{\max_{l\in\cbr{-k^{2/3},\ldots,k^{2/3}}}|\hat{W}_{k}-\hat{W}_{k+l}|\leq k^{4/9}}.
\]
By the properties of a Wiener process we know that $\max_{l\in\cbr{-k^{2/3},\ldots,k^{2/3}}}|\hat{W}_{k}-\hat{W}_{k+l}|$
has the tails decaying faster than $\exp\rbr{-t^{2}/(4k^{2/3})}$,
for $t$ large enough. Thus 
\[
1-q_{k}\leq \pr{\max_{l\in\cbr{-k^{2/3},\ldots,k^{2/3}}}|\hat{W}_{k}-\hat{W}_{k+l}|\geq k^{4/9}}\leq\exp\rbr{-k^{2/9}/4}.
\]
This quantity is summable thus the proof follows by the standard application of the Borel-Cantelli lemma. From now on, we will work conditionally on $\mathcal{A}_{N}$. Using
its first condition we have 
\begin{equation}
\tilde p_{N}\leq\log\pr{\forall_{\log N\leq n\leq N}2(\log N)^{2}+B_{n}\geq\hat{W}_{n} - A n^{1/2-\epsilon}}.\label{eq:estt1}
\end{equation}
We use the coupling techniques also for $\mathbb{P}$. Namely, by
\cite[Theorem 3.1]{Lifshits:2000fj} on a common probability space
(denoted still by $\mathbb{P}$), we have processes $\cbr{B_{k}}_{k\geq0}$,
distributed as the random walk from our theorem and $\cbr{\hat{B}_{t}}_{t\geq0}$
a Brownian motion which approximates $B$. Recalling (\ref{eq:timeChange})
we define 
\begin{equation}
\mathcal{B}_{N}\eqdef\cbr{\forall_{k\leq N}|B_{k}-\hat{B}_{f(k,\mu)}|\leq(\log N)^{2}}\cap\cbr{\forall_{k\leq N}\sup_{t\in[k,k+1]}(\hat{B}_{t}-\hat{B}_{k})\leq(\log N)^{2}}.\label{eq:inhomgenousApproximation}
\end{equation}
Applying \cite[Theorem 3.1]{Lifshits:2000fj} to the first term and standard considerations to the second one we obtain $\log\pr{\mathcal{B}_{N}^{c}}/\log N\to_{N\to+\infty}-\infty$.
We continue estimations of (\ref{eq:estt1}) as follows
\[
\tilde p_{N}\leq\pr{{\forall_{\log N\leq n\leq N}4(\log N)^{2}+\hat{B}_{f(n,\mu)}\geq\hat{W}_{n} - A n^{1/2-\epsilon}},\mathcal{B}_{N}}+\pr{\mathcal{B}_{N}^{c}}.
\]
We extend, in the piece-wise linear fashion, the function $f$ to
the whole line with respect to its first argument. This function is
non-decreasing and we denote its generalized inverse by $g(t,\mu)\eqdef\inf\cbr{s\geq0:f(s,\mu)\geq t}$.
We change to the continuous time (writing $t$ instead of $n$). By
the second and last condition of (\ref{eq:eventAN}) for some $C>0$
we have 
\[
\tilde p_{N}\leq\pr{\forall_{C\log N\leq t\leq N/C}5(\log N)^{2}+\hat{B}_{t}\geq\hat{W}_{g(t,\mu)} - A (g(t,\mu))^{1/2-\epsilon}}+\pr{\mathcal{B}_{N}^{c}}.
\]
Using two last conditions of (\ref{eq:eventAN}) one checks that $\mathcal{A}_N \subset \cbr{\forall_{t\geq C\log N}\hat{W}_{g(t,\mu)}\geq \hat{W}_{t/\qq(\ev{}(X_{1}^{2}))}-t^{4/9}}$
and thus conditionally on $\mathcal{A}_{N}$ we have 
\[
\tilde p_{N}-\pr{\mathcal{B}_{N}^{c}}\leq\pr{\forall_{C\log N\leq t\leq N/C}5(\log N)^{2}+\hat{B}_{t}\geq\hat{W}_{t/\qq(\ev{}(X_1^{2}))}-t^{4/9}-2A t^{1/2-\epsilon}}=:\hat{p}_{N}.
\]
Utilizing Theorem \ref{thm:BMExtended} one gets 
\[
\lim_{N\to+\infty}\frac{\log\hat{p}_{N}}{\log N}=-\gamma\rbr{\sqrt{\frac{\E W_{1}^{2}}{\qq(\ev{}(X_{1}^{2}))}}}.
\]
We recall that in our notation $\qq(\ev{}(X_{i}^{2}))$ is the same
as $\E\sbr{\E\rbr{B_1^{2}|\mu}}=\E B_{1}^{2}$ in the standard notation.
Recalling that $\pr{\mathcal{B}_{N}^{c}}$ is negligible we obtain
(\ref{eq:smallAim-1}).

Before passing further let us state a simple conditioning fact.
\begin{lem}
\label{lem:simpleFKG}Let $\cbr{T_{n}}_{n\geq0}$ be a random walk
and $\cbr{a_{n}}_{n\geq0}$ be a sequence. Then for any $N$ the law
$\pr{T_{N}\in\cdot|\forall_{n\in\cbr{1,\ldots,N}}T_{n}\geq a_{n}}$
stochastically dominates the law of $T_{N}$.
\end{lem}
Its proof following by inductive application of the Markov property is easy and thus skipped. 

We pass to the bound from below. We recall (\ref{eq:pN}). Our aim is to prove 
\begin{equation}
\liminf_{N\to+\infty}\frac{\log p_{N}}{\log N}\geq\begin{cases}
-\gamma & \text{ on }\mathcal{A}_{x},\\
-\infty & \text{ on }\mathcal{A}_{x}^{c}.
\end{cases}\label{eq:aim}
\end{equation}
We denote $K_{N}\eqdef\lfloor(\log N)^{6}\rfloor$ and $A_{N}\eqdef cK_{N}^{1/2}\log\log K_{N}$
($c$ is as in Lemma \ref{thm:veryFast} ). Utilizing the Markov property we obtain
\begin{align*}
\log p_N & \geq\log q_N+\log \hat p_N,
\end{align*}
where 
\begin{align*}
	q_N &:= \pr{{\forall_{n\in\cbr{0,\ldots,K_{N}}}x+B_{n}\geq W_{n}+f(n)},{B_{K_{N}}\geq A_{N}}, {B_{K_N}\leq N^{1/3}}},\\
	\hat p_N &\eqdef \inf_{x\in [A_N, N^{1/3}]}\pr{\forall_{n\in\cbr{K_N,\ldots,N}}x+B_{n}^{K_N}\geq W_{n}+f(n), x+B_N^{K_N} - W_N \in (aN^{1/2}, b N^{1/2})}, 
\end{align*}
where for $l$ we denote $B^l_k=B_k - B_l$, $k\geq l$. For $q_N$ we utilize Lemma \ref{lem:simpleFKG} as follows
\begin{equation} 
	q_N \geq \pr{{\forall_{n\in\cbr{0,\ldots,K_{N}}}x+B_{n}\geq W_{n}}}\pr{{B_{K_{N}}\geq A_{N}}} - \pr{B_{K_N}\geq N^{1/3}}.\label{eq:intermediateAims}
\end{equation}
Let us further denote $k_{N}\eqdef\lfloor(\log K_{N})^{6}\rfloor$ and $a_{N}\eqdef k_{N}^{1/2}\log\log k_{N}$.
Applying a similar procedure we get 
\[
	\pr{{\forall_{n\in\cbr{0,\ldots,K_{N}}}x+B_{n}\geq W_{n}+f(n)}} \geq \log\pr{B_{k_{N}}\geq a_{N}}+ \log p(x,0,k_{N})+\log p(a_{N},k_{N},K_{N}),
\]
where $p(x,k,l)\eqdef\pr{\forall_{n\in\cbr{k,\ldots,l}}x+B_n^k\geq W_{n}+f(n)}$. We will prove that 
\begin{equation}
\liminf_{N\to+\infty}\frac{\hat p_N}{\log N}\geq-\gamma,\quad\mathbb{Q}-a.s.\label{eq:tmpAim}
\end{equation}
Lemma \ref{lem:crudeLowerBound} and Lemma \ref{thm:veryFast} imply
\begin{equation}
\liminf_{N\to+\infty}\frac{\log p(x,0,k_{N})}{\log N}\geq\begin{cases}
0 & \text{ on }\mathcal{A}_{x}\\
-\infty & \text{ on }\mathcal{A}_{x}^{c}
\end{cases},\quad\text{and }\quad\liminf_{N\to+\infty}\frac{\log\pr{B_{K_{N}}\geq A_{N}}}{\log N}=0,\quad\mathbb{Q}-a.s.\label{eq:hardArgument}
\end{equation}
By simple scaling arguments we notice that (\ref{eq:tmpAim}) and (\ref{eq:hardArgument}) imply 
\[
\liminf_{N\to+\infty}\frac{\log p(a_{N},k_{N},K_{N})}{\log N}\geq0,\quad\text{and}\quad\liminf_{N\to+\infty}\frac{\log\pr{B_{k_{N}}\geq a_{N}}}{\log N}=0,\quad\mathbb{Q}-\text{a.s.}
\]
We notice that by assumption (A2) for large $N$ we have $\pr{B_{K_N}\geq N^{1/3}}\leq \exp(-N^{1/4})$. Thus this term is negligible in \eqref{eq:intermediateAims} and we get
\[
	\liminf_{N\to+\infty}\frac{\log q_N}{\log N}\geq\begin{cases}
	0 & \text{ on }\mathcal{A}_{x}\\
	-\infty & \text{ on }\mathcal{A}_{x}^{c}
	\end{cases}.
\]
This together with (\ref{eq:tmpAim}) implies \eqref{eq:aim}. For (\ref{eq:tmpAim}) we will apply coupling
arguments similar to the ones in the previous proof. We keep the notation
$(W,\hat{W})$ and $(B,\hat{B})$. We will also use the events of
(\ref{eq:eventAN}). Finally, we know that for some $\epsilon>0$ we have $f(n)\leq n^{1/2-\epsilon}/2$ for $n$ large enough. We set $a',b'$ such that $a<a'<b'<b$.  Conditionally on $\mathcal{A}_{N}$ for $N$ large enough 
\[
	\hat p_N \geq \pr{\forall_{K_{N}\leq n\leq N} A_N+B_{n}^{K_N}\geq\hat{W}_{n}+n^{1/2-\epsilon}/2, B_N^{K_N} - \hat W_N \in (a' N^{1/2}, b' N^{1/2})}.
\]
Further, recalling (\ref{eq:timeChange}) and (\ref{eq:inhomgenousApproximation}) for $a'<a''<b''<b'$  
we have
\[	\hat p_{N}\geq \pr{\begin{array}{l}\forall_{K_{N}\leq n\leq N}A_{N}/2+\hat{B}_{f(n,\mu)-f(K_{N},\mu)}\geq\hat{W}_{n}+n^{1/2-\epsilon}\\ \hat{B}_{f(N,\mu)-f(K_{N},\mu)}-\hat{W}_N \in(a''N^{1/2},b'' N^{1/2})\end{array}}-\pr{\mathcal{B}_{N}^{c}}.	
\]
Similarly as in the previous case the second term will be negligible.
Let $f_{N}(\cdot,\mu)$ be the piece-wise linearization of $\cbr{K_{N},\ldots,N}\ni n\mapsto f(n,\mu)-f(K_{N},\mu)$.
It is non-decreasing thus we may define its inverse by $g_{N}(t,\mu)\eqdef\inf\cbr{s\geq0:f_{N}(s,\mu)\geq t}$. We set $v=\qq(\ev{}(X_{1}^{2}))$ (we recall that $X_1 = B_1$)
\[
\mathcal{C}_{N}\eqdef\cbr{\forall_{t\geq0}|g_{N}(t,\mu)-t/v|\leq[(\log N)^{3}\vee t^{2/3}]}.
\]
We leave to the reader verifying that $1_{\mathcal{C}_{N}}\to1$ $\mathbb{Q}$-a.s. Now conditionally on $\mathcal{C}_N$ we have
\[
	\hat p_{N}\geq\pr{\forall_{v K_{N}/2\leq t\leq M_N}A_{N}/2+\hat{B}_{t}\geq\hat{W}_{g_{N}(t,\mu)} +(g_N(t,\mu))^{1/2-\epsilon}, \mathcal{D}_N }\\-\pr{\mathcal{B}_{N}^{c}},
\]
where $M_N = vN - N^{3/4}$ and
\[
	\mathcal{D}_N \eqdef \cbr{\forall_{M_N<t<vN+N^{3/4}} \hat{B_t} - \hat{W}_{g_N(t,\mu)} \in (a''N^{1/2}, b''N^{1/2})}.
\]
Using the third condition of (\ref{eq:eventAN}) and performing simple calculations we have
\[
\mathcal{A}_{N}\cap\mathcal{C}_{N}\subset\cbr{\forall_{t\geq\log N}\hat{W}_{g_{N}(t,\mu)}\leq\hat{W}_{t/v}+t^{4/9}+(\log N)^{3}}.
\]
Therefore on $\mathcal{A}_{N}\cap\mathcal{C}_{N}$, for $N$ large
enough, we get
\[
	\hat p_{N}\geq\pr{\forall_{vK_{N}/2\leq t\leq M_N}A_{N}/4+\hat{B}_{t}\geq\hat{W}_{t/v}+t^{4/9}, \mathcal{D}_N}\\-\pr{\mathcal{B}_{N}^{c}}.
\]
We choose $a''',b'''$ such that $a''<a'''<b'''<b''$  and apply the Markov property 
\begin{multline*}
	\pr{\forall_{vK_{N}/2\leq t\leq M_N}A_{N}/4+\hat{B}_{t}\geq\hat{W}_{t/v}+t^{4/9}, \mathcal{D}_N} \\
	\geq \pr{\forall_{K_{N}/2\leq t\leq M_N}A_{N}/4+\hat{B}_{t}\geq\hat{W}_{t/v}+t^{4/9}, \hat B_{M_N} - \hat W_{M_N/v}\in (a'''N^{1/2}, b'''N^{1/2})} \\
	\times \inf_{x\in (a'''N^{1/2}, b'''N^{1/2})}\pr{\mathcal{D_N}|\hat B_{M_N} - \hat W_{M_N/v}=x}.
\end{multline*}
It is easy to check that with high probability (with respect to $\mathbb{Q}$) the last term is bigger than $1/2$. 
	Recalling that $\pr{\mathcal{B}_{N}^{c}}$ is negligible and utilizing Theorem \ref{thm:BMExtended} we obtain (\ref{eq:tmpAim}). This together
with (\ref{eq:smallAim-1}) implies (\ref{eq:aimAim}) for the special
choice of the realization of $W$ (i.e. we worked with the measure
$\mathbb{Q}$ on which we had the coupling $(W,\hat{W})$). To remove
this assumption let us consider $l$ be the space with $\R$-valued
sequences with the product topology. Given any other probability measure
$\mathbb{P}$ supporting $W$ and $B$ we have $\pr{W\in\mathcal{A}}=\mathbb{Q}(W\in\mathcal{A})$
for any $\mathcal{A}$ in the Borel $\sigma$-field of $l$. One checks
that 
\[
\mathcal{A}_{0}\eqdef\cbr{g\in l:\lim_{N\to+\infty}\frac{\log\pr{\forall_{n\leq N}x+B_{n}\geq g_{n}+f(n)},B_N-W_N \in (aN^{1/2},bN^{1/2})}{\log N}=-\gamma}
\]
belongs to this $\sigma$-field. Now we have
\[
\pr{\mathcal{A}_{0}}=\qq(\mathcal{A}_{0})=1.
\]
This concludes the proof of Theorem \ref{thm:BRWBasicExtended}. We skip the proof of Theorem \ref{thm:BRWBasicExtended2}, it follows by rather simple modifications of the above proof.

We are still left with 
\begin{proof}[Proof of Fact \ref{fact:nonInfinity}] The first part of the fact is easy
e.g. by the Hsu\textendash Robbins theorem. For the second part let us consider first that $\sup S_{B}=+\infty$.
Then every step of $B$ can be bigger than the one of $W$ thus clearly
for any $N$ we have $\pr{\forall_{n\leq N}x+B_{n}\geq W_{n}|W}>0$.
Now, we assume $S\eqdef\sup S_{B}<+\infty$. For any fixed $N$ we
have $\pr{B_{1}\geq S-x/(2N)}>0$. Further, one verifies that 
\[
\cbr{\forall_{n\leq N}:B_{n}-B_{n-1}\geq S-x/(2N)}\subset\cbr{\pr{\forall_{n\leq N}x+B_{n}\geq W_{n}|W}>0}.
\]
Therefore, one obtains $\pr{\mathcal{A}_{x}}=1$. The second part
of the proof goes easily by contradiction. If the condition does not
hold then there exists $S$ and $\epsilon>0$ such that $\pr{W_{1}\geq S+\epsilon}>0$
and $\pr{B_{1}\geq S-\epsilon}=0$. From this we see that $\pr{W_{\lceil2x/\epsilon\rceil}\geq\lceil2x/\epsilon\rceil S+2x}>0$
while $\pr{B_{\lceil2x/\epsilon\rceil}\geq\lceil2x/\epsilon\rceil S}=0$.
\end{proof}

\section{Discussion and Open Questions}
\label{sec:Disussion-and-Open}

In the concluding section we discuss some open questions and further
areas of research.
\begin{itemize}
\item The function $\gamma$ introduced in Theorem \ref{thm:BMBasic} calls for
better understanding. We are convinced that it is strictly convex.
It should be possible to obtain its asymptotics when $\beta\to+\infty$, we expect that $\gamma(\beta)/\beta^2 \to C$, for $C>0$.
\item The qualitative results of our paper should hold in a much greater
generality. Let us illustrate that on an example. We expect that the
convergence in Theorem \ref{thm:BRWBasic} stays valid for any processes
$\cbr{W_{n}}_{n\in\N},\cbr{B_{n}}_{n\in\N}$ whose increments are
weakly correlated (for example with the exponential decay of correlations
like $\text{Cov}(W_{n+1}-W_{n},W_{k+1}-W_{k})\sim\exp(-c|n-k|)$). Similarly
the qualitative statement of Theorem \ref{thm:BMBasic} should be
valid if processes $\cbr{B_{t}}_{t\geq0},\cbr{W_{t}}_{t\geq0}$ are
diffusions without strong drift (possibly the proper condition to
assume is that the spectral gap is $0$).
\item The case $\beta=0$ in Theorem \ref{thm:BMBasic} is well-studied,
in particular it is known that conditioning a Brownian motion to stay
above the line has a repelling effect and such a process escapes to
infinity as $t^{1/2}$ as $t\to+\infty$. Our result $\gamma(\beta)>\gamma(0)$
for $\beta\neq0$ suggests that the repelling effect is stronger when
the disorder is present. Quantifying this effect would be an interesting
research question.
\end{itemize}

\subsection*{Acknowledgments }

We would like to thank prof. Anton Bovier, prof. Giambattista Giacomin, prof. Zhan Shi and prof. Ofer Zeitouni for the interest in our work and useful comments. Moreover, we would like to thank the reviewer for remarks, which led to improvement of the paper. 

\bibliographystyle{plain}

\end{document}